\documentclass[reqno]{amsart}

\usepackage{amsfonts}
\usepackage{amsmath,amssymb,amsxtra}
\usepackage{float}
\usepackage[colorlinks,linkcolor=RoyalBlue,anchorcolor=Periwinkle, citecolor=Orange,urlcolor=Green]{hyperref}
\usepackage[usenames,dvipsnames]{xcolor}
\usepackage{enumitem}
\usepackage{geometry} 
\usepackage{graphicx}
\usepackage{subfigure}
\usepackage{tikz}
\usepackage{url}
\usepackage[all]{xy}
\usepackage{mathdots}
\usepackage{yhmath}
\usepackage{amsmath}
\usepackage[final]{changes}
\numberwithin{figure}{section}

\newtheorem{thm}{Theorem}[section]
\newtheorem{lem}[thm]{Lemma}
\newtheorem{coro}[thm]{Corollary}
\newtheorem{main thm}[thm]{Main Theorem}
\newtheorem{prop}[thm]{Proposition}
\newtheorem{defn}[thm]{Definition}

\newtheorem{rmk}[thm]{Remark}

\theoremstyle{problem}

\theoremstyle{Observe}

\theoremstyle{plain}

\theoremstyle{plain}

\usetikzlibrary{arrows}

\numberwithin{equation}{section}

\def\<{\langle}
\def\>{\rangle}

\begin{document}

\title{Nowhere-zero flows on signed supereulerian graphs}
\author{Chao Wen}
\address{CW:
State Key Laboratory of Public Big Data,
School of Mathematics and Statistics,
Guizhou University,
550025, Guiyang,
China.}
\email{wenchao19960712@163.com}
\author{Qiang Sun}
\address{QS:
School of Mathematical Science, Yangzhou University, 225002, Yangzhou, China.}
\email{qsun1987@163.com}

\author{Chao Zhang}
\address{CZ:
State Key Laboratory of Public Big Data,
School of Mathematics and Statistics,
Guizhou University,
550025, Guiyang,
China.}
\email{zhangc@amss.ac.cn}


\begin{abstract}
In 1983, Bouchet conjectured that every flow-admissible signed graph admits a nowhere-zero $6$-flow.
We verify this conjecture for the class of flow-admissible signed graphs possessing a spanning even Eulerian subgraph, which includes as a special case all signed graphs with a balanced Hamiltonian circuit.
Furthermore, we show that this result is sharp by citing a known infinite family of signed graphs with a balanced Hamiltonian circuit that do not admit a nowhere-zero $5$-flow.
Our proof relies on a construction that transforms signed graphs whose underlying graph admits a nowhere-zero $4$-flow into a signed $3$-edge-colorable cubic graph.
This transformation has the crucial property of establishing a sign-preserving bijection between the bichromatic cycles of the resulting signed cubic graph and certain Eulerian subgraphs of the original signed graph.
As an application of our main result, we also show that Bouchet's conjecture holds for all signed abelian Cayley graphs.

\end{abstract}

\maketitle

\section{Introduction}

  All graphs in this paper are finite, loopless and may have multiple edges.
  Set $[a,b]=\{x\in \mathbb{Z}: a\leq x\leq b\}$.
  For basic notation and terminology which are not defined here, we refer to \cite{BM08,Z97}.
  A nowhere-zero flow is a way of assigning an orientation and a nonzero value from an abelian group $A$ to each edge of a graph, such that the Kirchhoff current law is satisfied at every vertex.
  This law requires that the sum of values flowing into a vertex equals the sum of values flowing out of it.
  The concept of integer flow was introduced by Tutte \cite{T49,T54} when he observed that each nowhere-zero $k$-flow on a plane graph corresponds to a $k$-face-coloring of it, and vice versa.
  Jaeger \cite{J88} further demonstrated that if a graph $G$ has a $k$-face-colorable $2$-cell embedding in an orientable surface, then it admits a nowhere-zero $k$-flow.
  Therefore, nowhere-zero flow and face coloring can be seen as dual concepts.
  Due to the duality between local tensions and flows on graphs embedded in nonorientable surfaces, Bouchet \cite{B83} systematically developed an analogous concept of a nowhere-zero flow using bidirected edges instead of directed ones in 1983.
  Since signed graphs provide a convenient language for describing such embeddings, the nowhere-zero flow on a signed graph is generally used to represent the nowhere-zero flow introduced by Bouchet.

  Bouchet \cite{B83} conjectured in 1983 that every flow-admissible signed graph admits a nowhere-zero $6$-flow, wherein he proved that such signed graphs admit a nowhere-zero $216$-flow.
  This question has attracted a lot of attention since then.
  In 1987, Z\'{y}ka \cite{Z87} improved Bouchet's results to nowhere-zero $30$-flow.
  Recently, Z\'{y}ka's results were improved by DeVos et al. \cite{DLLLZZ21} to nowhere-zero $11$-flow, which is the best current general approach to Bouchet's conjecture.

  Our work focuses on a specific class of such graphs.
  Recall that a graph $G$ is \textit{supereulerian} if it contains a spanning Eulerian subgraph.
  We introduce the concept of an \textit{even Eulerian signed graph}, defined as a signed Eulerian graph containing an even number of negative edges. It is a known result that every supereulerian graph admits a nowhere-zero $4$-flow.
  This fact, combined with a recent theorem by Luo et al. \cite{LMSZ25} which builds upon the results of Li et al. \cite{LLLZZ23}, provides a baseline for our investigation.

\begin{thm}[\cite{LMSZ25}]\label{thm:8-flow}
Let $(G, \sigma)$ be a flow-admissible signed graph. If $G$ admits a nowhere-zero $4$-flow, then $(G, \sigma)$ admits a nowhere-zero $8$-flow.
\end{thm}

An immediate consequence of Theorem~\ref{thm:8-flow} is that every flow-admissible signed supereulerian graph admits a nowhere-zero $8$-flow.
The main contribution of this paper is to improve this bound for signed supereulerian graphs that contain a spanning even Eulerian subgraph. We prove that this class of signed graphs admits a nowhere-zero $6$-flow, thereby verifying Bouchet's conjecture in this special case.\\

  \noindent
    \textbf{Theorem \ref{see-6-flow}.} \emph{Let $(G,\sigma)$ be flow-admissible.
      If $(G,\sigma)$ has a spanning even Eulerian subgraph, then $(G,\sigma)$ admits a nowhere-zero $6$-flow.}\\

      If the spanning even Eulerian subgraph is a balanced Hamiltonian circuit, then the following theorem holds. \\

    \noindent
    \textbf{Theorem \ref{BH-6-flow}.} \emph{Let $(G,\sigma)$ be flow-admissible.
      If $(G,\sigma)$ has a balanced Hamiltonian circuit, then $(G,\sigma)$ admits a nowhere-zero $6$-flow.}\\

    Consider the signed cubic graph $(G_{n},\sigma_{n})$ derived from an even circuit \(C_{2n}\), where $n$ is odd positive integer.
    This signed graph is constructed by replacing every second edge with a pair of parallel edges and assigning a signature such that all single edges are positive, and exactly one edge in each pair of parallel edges is negative.
    Fig. \ref{BH has no 5-flow} illustrates the signed graph $(G_{3}, \sigma_{3})$.
    In our figures, negative edges are depicted by dashed lines.
    Note that $(G_{n},\sigma_{n})$ contains a balanced Hamiltonian circuit which is a spanning even Eulerian subgraph.
    M\'{a}\v{c}ajov\'{a} et al. \cite{MS15} and Schubert et al. \cite{SS15} independently proved that $(G_{n},\sigma_{n})$ admits a nowhere-zero $6$-flow but does not admit any nowhere-zero $5$-flow.
    Therefore, the value $6$ in Theorems \ref{see-6-flow} and \ref{BH-6-flow} is optimal.

    In order to prove Theorems \ref{see-6-flow} and \ref{BH-6-flow}, we introduce a method reduces the general case to the cubic case.
    More precisely, we construct a signed $3$-edge-colorable cubic graph from a signed $4$-NZF-admissible graph, where a graph is {\it $4$-NZF-admissible} if it admits a nowhere-zero $4$-flow.
    Note that every $3$-edge-colorable cubic graph is $4$-NZF-admissible.
    Utilizing this method, we prove the following two theorems.\\

    \noindent
    \textbf{Theorem \ref{equi-relation}.} \emph{Let $k$ be a positive integer. Then the following statements are equivalent:}

    {\rm(1)} \emph{Every flow-admissible signed $4$-NZF-admissible graph admits a nowhere-zero $k$-flow;}

    {\rm(2)} \emph{Every flow-admissible signed $3$-edge-colorable cubic graph admits a nowhere-zero $k$-flow.}\\

    For a specific class of signed $4$-NZF-admissible graphs, known as signed supereulerian graphs, and a specific class of signed $3$-edge-colorable cubic graphs, referred to as signed Hamiltonian cubic graphs, we present the following theorem.\\

    \noindent
    \textbf{Theorem \ref{equi-s-h-hc}.} \emph{
       Let $k$ be a positive integer. Then the following statements are equivalent:}

      {\rm(1)} \emph{Every flow-admissible signed supereulerian graph admits a nowhere-zero $k$-flow;}

      {\rm(2)} \emph{Every flow-admissible signed Hamiltonian graph admits a nowhere-zero $k$-flow;}

      {\rm(3)} \emph{Every flow-admissible signed Hamiltonian cubic graph admits a nowhere-zero $k$-flow.}\\


    Moreover, we apply Theorem \ref{BH-6-flow} to prove the following theorem for signed abelian Cayley graphs, which are a class of signed Hamiltonian graphs.\\

   \noindent
    \textbf{Theorem \ref{AC-6-flow}.} \emph{Every flow-admissible signed abelian Cayley graph admits a nowhere-zero $6$-flow.}\\

    The value $6$ is optimal, as there exists a signed abelian Cayley graph without any nowhere-zero $5$-flow, as shown in Fig. \ref{AC has no 5-flow}.

   \begin{figure}[h]
  \centering
   \begin{minipage}[b]{.5\textwidth}
    \centering
    \begin{tikzpicture}[scale=0.6]
    \draw (0,0) circle (2);

     \draw[dashed] (-1,1.75) arc (0:-60:2);
     \draw[dashed] (1,1.75) arc (180:240:2);
    \draw[dashed] (1,-1.75) arc (60:120:2);
   \fill (-1,1.75) circle (.1);
   \fill (-1,-1.75) circle (.1);
   \fill (1,1.75) circle (.1);
   \fill (1,-1.75) circle (.1);
   \fill (-2,0) circle (.1);
   \fill (2,0) circle (.1);

    \end{tikzpicture}
   \caption{$(G_{3},\sigma_{3})$.}\label{BH has no 5-flow}
     \end{minipage}%
     \begin{minipage}[b]{.5\textwidth}
    \centering
  \begin{tikzpicture}[scale=0.6]


    \draw[solid] (-2,2)--(2,2);

    \draw[dashed] (-2,2)--(-2,-2) ;

    \draw[solid] (-2,-2)--(2,-2);

    \draw[solid] (2,2)--(2,-2);

    \draw[dashed] (-1,1)--(1,1);

    \draw[solid] (-1,1)--(-1,-1);

    \draw[solid] (-1,-1)--(1,-1) ;

    \draw[solid] (1,1)--(1,-1);

    \draw [solid](-2,2)--(-1,1) ;

     \draw[solid] (-2,-2)--(-1,-1);

    \draw [dashed](2,-2)--(1,-1) ;

     \draw[solid] (2,2)--(1,1);

   \fill (-2,2) circle (.1);
   \fill (-2,-2) circle (.1);
   \fill (2,-2) circle (.1);
   \fill (2,2) circle (.1);

    \fill (-1,1) circle (.1);
   \fill (-1,-1) circle (.1);
   \fill (1,-1) circle (.1);
   \fill (1,1) circle (.1);

  \end{tikzpicture}
   \caption{A flow-admissible signed abelian Cayley graph without any nowhere-zero $5$-flow.}\label{AC has no 5-flow}
   \end{minipage}%
\end{figure}

  Inspired by the characterization of the flow number of signed Eulerian graphs \cite{MS17}, we characterize the flow number of a class of signed abelian Cayley graphs.
  The \textit{flow number} of $(G,\sigma)$, denoted by $\Phi(G,\sigma)$, is the minimum $k$ such that $(G,\sigma)$ admits a nowhere-zero $k$-flow.
  Let $E_{N}(G, \sigma)$ denote the set of negative edges in $(G, \sigma)$.
  \\

    \noindent
    \textbf{Theorem \ref{c-h-dec}.}
      \emph{Let $A$ be a finite abelian group of odd order and $\Gamma=Cay(A,S)$ is connected.
     If $(\Gamma,\sigma)$ is flow-admissible, then}

      \emph{{\rm(1)} $\Phi(\Gamma,\sigma)=2$ if and only if $\left| E_{N}(\Gamma,\sigma)\right|$ is even;}

      \emph{{\rm(2)} $\Phi(\Gamma,\sigma)=3$ if and only if $\left| E_{N}(\Gamma,\sigma)\right|$ is odd and $\frac{\left| S \right|}{2}\geq 3$;}

      \emph{{\rm(3)} $\Phi(\Gamma,\sigma)=4$ if and only if $\left| E_{N}(\Gamma,\sigma)\right|$ is odd and $\frac{\left| S \right|}{2}=2$.}
\\

  The organization of the rest of the paper is as follows.
  Basic notation and terminology are introduced in Section 2.
  In Section 3, we present the method that derives a signed $3$-edge-colorable cubic graph from a signed $4$-NZF-admissible graph.
  This section also includes the proofs of Theorems \ref{equi-relation}, and \ref{equi-s-h-hc}.
  Section 4 presents the proofs of Theorems \ref{see-6-flow} and \ref{BH-6-flow}, which establish sufficient conditions for a signed supereulerian graph to admit a nowhere-zero $6$-flow.
  As an application of Theorem \ref{BH-6-flow}, Theorem \ref{AC-6-flow} is proved in Section 5, which discusses signed abelian Cayley graphs.
  Additionally, Section 5 provides the characterization of the flow number of abelian Cayley graphs with an odd number of vertices, as stated in Theorem \ref{c-h-dec}.

\section{Notation and terminology}
  We write $G$ for a {\it graph}, with its vertex set and edge set denoted by $V(G)$ and $E(G)$, respectively.
  A {\it circuit} is a connected $2$-regular graph.
  A graph $G$ is said to be {\it even} if every vertex of $G$ has an even degree.
  A graph $G$ is called an \textit{Eulerian graph} if it is both connected and even.
  A graph $G$ is called {\it supereulerian} if it contains a spanning Eulerian subgraph.
  Specifically, a {\it Hamiltonian} graph is a supereulerian graph that contains a spanning circuit.

  A \textit{signed graph} is defined as $(G, \sigma)$, where $G$ is the underlying graph and $\sigma: E(G) \rightarrow \{\pm 1\}$ is a signature assigning a sign to each edge.
  An edge $e$ of $(G, \sigma)$ is \textit{positive} if \(\sigma(e) = +1\); otherwise, it is \textit{negative}.
  Recall that $E_{N}(G, \sigma)$ denote the set of negative edges in $(G, \sigma)$.
  A signed graph $(G, \sigma)$ is \textit{all-positive} if $E_{N}(G, \sigma) = \emptyset$.
  In this paper, ordinary graphs are considered as all-positive signed graphs.
  Let $F$ be a subgraph of $G$.
  The \textit{sign of $F$}, denoted by $\sigma(F)$, is the product of the signs of its edges.
  Specifically, let $\sigma(F) = +1$ if $E(F) = \emptyset$.
  A circuit $C$ is \textit{balanced} if $\sigma(C) = +1$, and \textit{unbalanced} otherwise.
  A signed graph \((G, \sigma)\) is called \textit{balanced} if there is no unbalanced circuit in \((G, \sigma)\), and \textit{unbalanced} otherwise.

  \textit{Switching} is an operation on a signed graph.
  For a vertex $v \in V(G)$, \textit{switching at $v$} negates the sign of each edge incident with $v$.
  For a vertex set $U$, \textit{switching at $U$} means switching all vertices in $U$.
  It is worth noting that switching does not change the sign of any circuit.
  If the signed graph $(G,\sigma')$ is obtained from $(G,\sigma)$ by a sequence of switchings, then we say that $(G,\sigma')$ is {\it switching equivalent} to $(G,\sigma)$.
  Switching equivalence is an equivalence relation.

  Two signed graphs $(G,\sigma)$ and $(H,\pi)$ are \textit{isomorphic}, denoted by $(G,\sigma)\cong(H,\pi)$ if there is an isomorphism $f$ from $G$ to $H$ such that for any circuit $C$ in $G$, $\sigma(C)=\pi(f(C))$.
  It is easy to see that if $f$ is an isomorphism from $G$ to $H$ such that $\sigma(e)=\pi(f(e))$ for any $e \in E(G)$, then $f$ is an isomorphism from $(G,\sigma)$ to $(H,\pi)$.

  Let $G_{1}$ be a subgraph of $G$.
  It is convenient to denote the signed graph $(G_{1},\sigma\mid_{E(G_{1})})$ by $(G_{1},\sigma)$, where $\sigma\mid_{E(G_{1})}$ is a restriction of $\sigma$ to $E(G_{1})$.

  A {\it signed circuit} is a signed graph that belongs to one of the following three types:

  (1) A balanced circuit;

  (2) A short barbell, which is the union of two unbalanced circuits that meet at a single vertex;

  (3) A long barbell, which is the union of two disjoint unbalanced circuits with a path that meets the circuits only at its ends.

  For an edge $e$ with two ends $u$ and $v$, it can be regarded as two half edges $h^{u}_{e}$ and $h^{v}_{e}$, where $h^{u}_{e}$ is incident with $u$ and $h^{v}_{e}$ is incident with $v$.
  Let $H(G)$ be the \textit{set of all half edges of $G$}, and $H_{G}(u)$ be the \textit{set of all half edges incident with $u$}.
  An {\it orientation} of $(G,\sigma)$ is a mapping $\tau : H(G)\rightarrow \{\pm 1\}$ such that $\tau(h^{u}_{e})\tau(h^{v}_{e})=-\sigma(e)$ for each edge $e\in E(G)$.
  For a half edge $h^{u}_{e}\in H(G)$, we say $h^{u}_{e}$ is {\it oriented away from $u$} if $\tau(h^{u}_{e})=+1$; otherwise $h^{u}_{e}$ is {\it oriented toward $u$}.

  \begin{defn}
    Let $(G,\sigma)$ be a signed graph, $A$ be an abelian group, and $\tau$ be an orientation of $(G,\sigma)$. Let $f: E(G)\rightarrow A$ be a function, and $k\geq2$ be an integer.

    {\rm(1)} For each vertex $v\in V(G)$, the {\it boundary} of $f$ at $v$ is $$\partial f(v)=\sum\limits_{h\in H_{G}(v)}\tau(h)f(e_{h}),$$ where $e_{h}$ is the edge of $G$ containing the half edge $h$.

    {\rm(2)} The \textit{support} of $f$, denoted by $supp(f)$, is the set of edges $e$ for which $f(e)$ is not equivalent to the identity element of $A$.

    {\rm(3)} Let $A = \mathbb{Z}$. Then the ordered pair $(\tau, f)$ is a \textit{$k$-flow} of $(G, \sigma)$ if $\partial f(v) = 0$ for each $v \in V(G)$ and $|f(e)| < k$ for each $e \in E(G)$. A $k$-flow $(\tau, f)$ is a \textit{nowhere-zero $k$-flow} if $supp(f) = E(G)$.

    {\rm(4)} Let $A = \mathbb{Z}_k$. Then the ordered pair $(\tau, f)$ is called a \textit{$\mathbb{Z}_k$-flow} of $(G, \sigma)$ if $\partial f(v) = 0$ for each vertex $v$. A $\mathbb{Z}_k$-flow $(\tau, f)$ is a \textit{nowhere-zero $\mathbb{Z}_k$-flow} if $supp(f) = E(G)$.
  \end{defn}

  For convenience, we abbreviate ``nowhere-zero $k$-flow" as $k$-NZF and ``nowhere-zero $\mathbb{Z}_{k}$-flow" as $\mathbb{Z}_{k}$-NZF.
  If the orientation is understood from the context, we use $f$ instead of $(\tau, f)$ to denote a flow.

  Switching at a vertex $v$ only reverses the directions of the half edges incident with $v$, while the directions of other half edges and the flow values of all edges remain unchanged.
  Thus, if $(G, \sigma)$ is switching equivalent to $(G, \sigma')$ and $(G, \sigma)$ admits a $k$-NZF, then $(G, \sigma')$ also admits a $k$-NZF.

  A signed graph is considered {\it flow-admissible} if it admits a $k$-NZF for some integer $k$.
  The following characterization of flow-admissible signed graphs can be found in \cite{B83, MS17}.
  \begin{prop}\cite{MS17}\label{equi-flow-admissible}
   The following statements are equivalent for every connected unbalanced signed graph $(G,\sigma)$:

  {\rm(a)} $(G,\sigma)$ is flow-admissible.

  {\rm(b)} The edges of $G$ can be covered with signed circuits.

  {\rm(c)} $(G,\sigma)$ has no edge $e$ such that $(G \setminus e, \sigma)$ has a balanced component.
  \end{prop}

  Based on statements (a) and (b) in Proposition \ref{equi-flow-admissible}, we have the following lemma.
  \begin{lem}
    Let $\mathcal{S}(G, \sigma)$ be the set of all signed circuits in the signed graph $(G, \sigma)$. Then $(G, \sigma)$ is flow-admissible if and only if

$$
\bigcup_{(C, \sigma) \in \mathcal{S}(G, \sigma)} E(C) = E(G),
$$
i.e., every edge of $G$ is covered by signed circuits.
  \end{lem}

  There is a direct corollary as follows.

  \begin{coro}
    The signed graph $(G,\sigma)$ is flow-admissible if and only if every edge of $G$ is contained in a flow-admissible signed subgraph of $(G,\sigma)$.
  \end{coro}

    For any ordinary graph $G$, we define $E_{G}(v)=\{e\in E(G) : e\ is\ incident\ with\ v\}$, where $v$ is a vertex in $V(G)$.
    For $F \subseteq E_{G}(v)$, we denote by $G_{[v;F]}$ the graph obtained from $G$ by splitting all edges of $F$ away from $v$ and adding a new vertex $v'$ as the end of these edges.

    Note that, for a signed graph $(G,\sigma)$, the signature $\sigma$ is a function defined on $E(G)$, and splittings do not change the edge set.
    Thus, $(G_{[v;F]},\sigma)$ is a signed graph obtained from $(G,\sigma)$ by performing a splitting at $v$ with respect to $F$.
    Furthermore, if $(G_{[v;F]},\sigma)$ admits a $k$-NZF, then so does $(G,\sigma)$.


    \section{Signed $4$-NZF-admissible graphs and $3$-edge-colorable cubic graphs}

    In this section, a method is developed for deriving a signed $3$-edge-colorable cubic graph from a signed $4$-NZF-admissible graph.
    Using this method, we establish an equivalence in the admission of $k$-NZF between signed $4$-NZF-admissible graphs and signed $3$-edge-colorable cubic graphs, as well as between signed supereulerian graphs and signed Hamiltonian cubic graphs.
    By applying these relationships, we show that every flow-admissible signed $4$-NZF-admissible graph admits a $10$-NZF, and every flow-admissible signed supereulerian graph admits an $8$-NZF.
    Furthermore, we apply these relationships to prove that signed supereuler graphs with a spanning even Eulerian graph admit a $\mathbb{Z}_{4}$-NZF.

    To achieve these results, we first introduce and explore several properties of $4$-NZF-admissible graphs.
    The following theorem illustrates how two $2$-flows contribute to our understanding of the structure of graphs that are $4$-NZF-admissible.

    \begin{thm}\cite{Z97}
      Let $G$ be a graph and $k_{1}$, $k_{2}$ be two integers.
      Then $G$ admits a nowhere-zero $k_{1}k_{2}$-flow if and only if $G$ admits a $k_{1}$-flow $f_1$ and a $k_{2}$-flow $f_2$ such that $supp(f_{1})\cup supp(f_{2}) = E(G)$.
    \end{thm}

    Hence, the graph $G$ is $4$-NZF-admissible if and only if $G$ admits a $2$-flow $f_1$ and a $2$-flow $f_2$ such that $supp(f_{1})\cup supp(f_{2}) = E(G)$.
    Note that if $f$ is a $2$-flow in $G$, then $supp(f)$ induces an even subgraph in $G$.
    Consequently, a $4$-NZF-admissible graph can be covered by two even subgraphs.

    We need more notation and terminology.

    To {\it contract} an edge $e$ of a graph $G$ means to delete the edge $e$ and then identify its ends.
    The resulting graph is denoted by $G/e$.
    For $S\subseteq E(G)$, let $G/S$ denote the graph obtained from $G$ by contracting all edges of $S$.

    Let $C(G)$ be the set of components of graph $G$.
    The degree of a vertex $v$ in the graph $G$, denoted by $d_{G}(v)$.

    Let $X$ and $Y$ be two disjoint vertex sets of $G$.
    We denote by $E_{G}[X,Y]$ the set of edges of $G$ with one end in $X$ and the other end in $Y$.
    For a subgraph $H$ of $G$, denote the boundary of $H$, $E_{G}[V(H),V(G)\setminus V(H)]$, by $\partial(H)$.

    The following lemma presents the method  for deriving a signed $3$-edge-colorable graph from a signed
    $4$-NZF-admissible graph, while preserving certain properties.

    \begin{lem}\label{lem-4-into-3}
      Let $G$ be a $4$-NZF-admissible graph, with $f_{1}$ and $f_{2}$ being two $2$-flows on $G$ such that $supp(f_{1})\cup supp(f_{2}) = E(G)$.

      For a signed $4$-NZF-admissible graph $(G,\sigma)$,
      there exists a signed $3$-edge-colorable graph $(G',\sigma')$ such that the following statements hold:

      {\rm (1)} If $(G,\sigma)$ is flow-admissible, then $(G',\sigma')$ is flow-admissible;

      {\rm (2)} Let $H_{1}$ be a spanning subgraph of $G$ with edge set $supp(f_{1})$.
      There exists a $2$-factor $J$ of $G'$ such that there is a bijection $f:  C(H_{1}) \rightarrow  C(J) $, and for any $I \in C(H_{1})$, we have $\sigma(I)=\sigma'(f(I))$;

      {\rm (3)} Let $S=E(G')\setminus E(G)$.
      For any edge $e\in S$, we have $\sigma'(e)=+1$.
      Furthermore, $(G'/S,\sigma'\mid _{E(G'/S)})\cong(G,\sigma)$.
    \end{lem}

    \begin{proof}

      Since a cubic graph is $3$-edge-colorable if and only if it has a $2$-factor and each component of the $2$-factor forms an even circuit, our objective is to derive an even circuit from every component of $H_{1}$.
      Meanwhile, we must ensure that the flow-admissible property is maintained if $(G,\sigma)$ is flow-admissible.
      Thus, we may always assume that $(G,\sigma)$ is flow-admissible.

      Since a component $I$ of $H_{1}$ is Eulerian, $I$ has an Euler tour, denoted by $T$.
      Let $T=v_{0}e_{1}v_{1}e_{2}\cdots v_{k}e_{0}v_{0}$.
      We will use a series of splittings such that $I$ converts into a circuit.
      Let $G^{1}=G$.
      For  $x \in [2,k]$, $G^{x}={G^{x-1}}_{[v_{x};\{e_{x},e_{x+1}\}]}$ if $\left| E_{G_{x-1}}(v_{x})\cap E(I) \right|>2$, otherwise $G^{x}=G^{x-1}$, where $e_{k+1}=e_{0}$.
      This iteration will be carried out $k-1$ times,
      and the resulting graph is $G^{k}$.
      Meanwhile, the resulting signed graph is $(G^{k},\sigma)$.
      Let $I^{k}$ be the subgraph of $G^{k}$ induced by edge set $E(I)$.
      Since $I$ is Eulerian and $T$ is the Euler tour of $I$, we have $I^{k}$ is a circuit.
      Note that $(G^{k},\sigma)$ may not be flow-admissible.
      We will add some positive edges to ensure the property of flow-admissibility.

      For any $v\in V(G)$, let $\{v'_{0},v'_{1}, v'_{2},\cdots, v'_{l} \}$, where $l\geq 0$ and $v'_{0}=v$, be a vertex set whose elements are obtained by splitting $v$.
      If $l=1$, we add two positive edges (multiple edges) to connect $v'_{0}$ and $v'_{1}$.
      If $l>1$, we add $\frac{l(l+1)}{2}$ positive edges to $(G^{k},\sigma)$ such that $\{v'_{0}, v'_{1}, v'_{2},\cdots, v'_{l} \}$ induce an all-positive complete graph.
      Denote the new signed graph by $(G_{1},\sigma_{1})$.
      Since all-positive digons and all-positive complete graphs are flow-admissible, the added positive edges are covered by flow-admissible subgraphs.
      Next, we need to verify that every edge in $E(G)$ is also covered by a flow-admissible subgraph.

      More precisely, since $(G,\sigma)$ is flow-admissible, every edge in $E(G)$ is covered by a signed circuit in $(G,\sigma)$.
      Thus, our goal is to show that each signed circuit in $(G,\sigma)$ can be extended to a signed circuit in $(G_{1},\sigma_{1})$.
      Let $(C,\sigma)$ be a signed circuit of $(G,\sigma)$ and $(C',\sigma_{1})$ be subgraph of $(G_{1},\sigma_{1})$ induced by $E(C)$.
      If $(C,\sigma)\cong(C',\sigma_{1})$, then we are done.
      Hence, we may always assume that $(C,\sigma)$ is not isomorphic to $(C',\sigma_{1})$, where there are three cases as follows.\\

      \textbf{Case 1.} $(C,\sigma)$ is a balanced circuit in $(G,\sigma)$.

      Since $C'$ is obtained from $C$ by a sequence of splittings and $C$ is a circuit, the vertices of $C'$ have degree $2$ or $1$.
      Let $v_{1}$ be a vertex of degree $1$ in $C'$.
      Then $v_{1}$ is split from a vertex $v$ in $C$.
      Since $v$ has degree $2$ in $C$, there is another vertex, say $v_{2}$, in $C'$ that is also split from $v$.
      It is easy to see that $d_{C'}(v_{2})=1$.
      Note that there is a positive edge $v_{1}v_{2}$ in $(G_{1},\sigma_{1})$.
      Then we add the positive edge $v_{1}v_{2}$ to $C'$, and we still denote the resulting graph by $C'$.
      We repeat this operation until there are no vertices of degree $1$ in $C'$.
      Then $(C',\sigma_{1})$ is a balanced circuit in $(G_{1},\sigma_{1})$.\\

      \textbf{Case 2.} $(C,\sigma)$ is a short barbell in $(G,\sigma)$.

      For two vertices of $C'$ that are split from a vertex of degree $2$ in $C$, we can add a positive edge to connect them, as in Case 1.
      Hence, for convenience, we may assume that only the vertex of degree $4$ in $C$ was split.
      Let $v$ be the vertex of degree $4$ in $C$, which split into several vertices in $C'$.\\

      \textbf{Subcase 2.1.} The vertex $v$ has been split into two vertices $v_{1}$ and $v_{2}$ in $C'$.

      We only need to consider the following two cases.
      Namely, $d_{C'}(v_{1})=d_{C'}(v_{2})=2$, or one of $d_{C'}(v_{1})$ and $d_{C'}(v_{2})$ is $3$ and the other is $1$.

      If $d_{C'}(v_{1})=d_{C'}(v_{2})=2$, then $(C',\sigma_{1})$ is either a balanced circuit or a union of two unbalanced circuits.
      For the first case, $(C',\sigma_{1})$ is already a signed circuit.
      For the second case, adding a positive edge $v_{1}v_{2}$ to $(C',\sigma_{1})$, the resulting signed subgraph is a long barbell in $(G_{1},\sigma_{1})$.

      If one of $d_{C'}(v_{1})$ and $d_{C'}(v_{2})$ is $3$ and the other is $1$, then we add a positive edge $v_{1}v_{2}$ of $(G_{1},\sigma_{1})$ to $(C',\sigma_{1})$.
      The resulting signed subgraph is a short barbell in $(G_{1},\sigma_{1})$.\\

      \textbf{Subcase 2.2.} The vertex $v$ has been split into three vertices $v_{1}$, $v_{2}$ and $v_{3}$ in $C'$.

      Since $d_{C}(v)=4$, there is a vertex $v_{i}$ that has degree $2$ in $C'$, $i\in [1,3]$, say $v_{1}$.
      Then $d_{C'}(v_{2})=d_{C'}(v_{3})=1$.
      Then either $(C',\sigma_{1})$ is a path with positive sign or a union of a path with negative sign and an unbalanced circuit.

      For the first case, adding a positive edge $v_{2}v_{3}$ to $(C',\sigma')$, the resulting signed subgraph is a balanced circuit in $(G_{1},\sigma_{1})$.

      For the second case, adding two positive edges $v_{1}v_{2}$ and $v_{2}v_{3}$ to $(C',\sigma_{1})$, the resulting signed subgraph is a long barbell in $(G_{1},\sigma_{1})$.\\

      \textbf{Subcase 2.3.} The vertex $v$ has been split into four vertices $v_{1}$, $v_{2}$, $v_{3}$ and $v_{4}$ in $C'$.

      Since $d_{C}(v)=4$, we have $d_{C'}(v_{1})=d_{C'}(v_{2})=d_{C'}(v_{3})=d_{C'}(v_{4})=1$.
      Hence, $(C',\sigma_{1})$ is a union of two paths with negative sign, say $(P_{1},\sigma_{1})$ and $(P_{2},\sigma_{1})$.
      Without loss of generality, let $\{v_{1}, v_{2}\}\subseteq V(P_{1})$, and $\{v_{3},v_{4}\}\subseteq V(P_{2})$.
      Then $v_{1}$ ,$v_{2}$ ,$v_{3}$ and $v_{4}$ are ends of $P_{1}$ and $P_{2}$, respectively.
      By adding two positive edges $ v_1v_3 $ and $ v_2v_4 $ to $ (C', \sigma_1) $, the resulting signed subgraph forms a balanced circuit in $ (G_1, \sigma_1) $.
      \\

      \textbf{Case 3.} $(C,\sigma)$ is a long barbell in $(G,\sigma)$.

      For convenience, we assume that only the vertices of degree $3$ in $ C $ have been split.
      By symmetry, we can further assume that only one vertex of degree $3$, say $ v $, has been split.

      Let $(C^{*},\sigma)$ be the unbalanced circuit in $(C,\sigma)$ that contained $v$ as a vertex, and $(P^{*},\sigma)$ be the path in $(C,\sigma)$ that meets the unbalanced circuits only at its ends.
      Let $(L,\sigma_{1})$ be a subgraph of $(C',\sigma_{1})$, where $L$ is induced by $E(C^{*})\cup E(P^{*})$.\\

      \textbf{Subcase 3.1.} The vertex $v$ has been split into two vertices $v_{1}$ and $v_{2}$ in $C'$.

      Since $ d_C(v) = 3 $, there exists a vertex $ v_i $ of degree 2 in $ C' $, where $ i \in [1, 2] $, say $ v_1 $. Consequently, the degree of $ v_2 $ is 1.
      The signed graph $ (L, \sigma_1) $ is either a path or a union of an unbalanced circuit and a path.
      Therefore, adding a positive edge $ v_1v_2 $ to $ (C', \sigma_1) $, the resulting signed subgraph forms a long barbell in $ (G_1, \sigma_1) $.\\

      \textbf{Subcase 3.2.} The vertex $v$ has been split into three vertices $v_{1}$, $v_{2}$ and $v_{3}$ in $C'$.

      Since $d_{C}(v)=3$, we have $d_{C}(v_{1})=d_{C}(v_{2})=d_{C}(v_{3})=1$.
      Then $(L,\sigma_{1})$ is a union of a path and a path with negative sign, say $(P_{1},\sigma_{1})$ and $(P_{2},\sigma_{1})$.
      Without loss of generality, let $ v_1 $ be an end in $ P_1 $, and let $ v_2 $ and $ v_3 $ be ends in $ P_2 $.
      Then, when we add two positive edges $v_{1}v_{2}$ and $v_{1}v_{3}$ to $(C',\sigma_{1})$, the resulting signed subgraph is a long barbell in $(G_{1},\sigma_{1})$.
      \\

      Since every edge of $(G_{1},\sigma_{1})$ is contained in a flow-admissible signed subgraph, $(G_{1},\sigma_{1})$ is also flow-admissible.

      Next, we aim to transform all the vertices in $V(I^{k})$ into vertices of degree $3$ through a series of blow up.
      For technical reasons, a digon is considered a circuit of length $2$, denoted by $C_{2}$.
      For a vertex $v\in V(I^{k})$, we replace $v$ by an all-positive circuit $(C_{v},+)$ of length $d(v)$ and define the incidence relation between the edges of $E_{G_{1}}(v)$ and the vertices of $(C_{v},+)$ as follows.
      Let the two edges in $E(I^{k})\cap E_{G_{1}}(v)$ be incident with $v_{1}$ and $v_{2}$, respectively, where $v_{1}$ and $v_{2}$ are adjacent in $C_{v}$.
      Then, $E(I^{k})$ combined with $E(C_{v})\setminus \{v_{1}v_{2}\}$ can be extended to form a new circuit that contains all vertices of $I^{k}$ and $C_{v}$.
      Note that $\left| E_{G_{1}}(v)\setminus E(I^{k})\right|=\left| V(C_{v})\setminus\{v_{1},v_{2}\}\right|=d_{G_{1}}(v)-2$.
      Let $\varphi$ be an arbitrary bijection from $E_{G_{1}}(v)\setminus E(I^{k})$ to $V(C_{v})\setminus\{v_{1},v_{2}\}$.
      Then an edge $e \in E_{G_{1}}(v)\setminus E(I^{k})$ is incident with a vertex $v\in V(C_{v})\setminus\{v_{1},v_{2}\}$ if and only if $\varphi(e)=v$.
      Since $(C_{v},+)$ is a balanced circuit, every edge of $E(C_{v})$ is covered by a signed circuit.
      For any signed circuit in $(G_{1},\sigma_{1})$ containing $v$, it is easy to verify that replacing vertex $v$ by an all-positive circuit still maintains flow-admissibility.
      Therefore, every edge of the resulting signed graph is contained in a flow-admissible subgraph, i.e., the resulting signed graph is flow-admissible.
      We repeat this operation for all vertices of $I^{k}$.
      Denote the new circuit obtained from $ I^k $ by $ I_2 $, and the resulting signed graph by $ (G_2, \sigma_2) $.
      It is easy to see that every vertex in $V(I_{2})$ has degree $3$ in $G_{2}$ and $(G_{2},\sigma_{2})$ is flow-admissible.

      It remains to prove that $I_{2}$ is an even circuit.
      Note that any edge which is incident with two vertices of $I_{2}$ contributes an even number of vertices to $I_{2}$.
      Thus, in graph $G_{2}$, we only need to consider the number of edges that have only one end in $I_{2}$.
      We need to note that the series of operations we performed to transform $I$ into $I_{2}$ do not change the number of edges connecting this component to the outside.
      Hence, we only need to consider the number of edges that have only one end in $I$ in graph $G$, i.e., the number of elements in the boundary $\partial_{G}(I)$.
      Recall that $f_{1}$ and $f_{2}$ are two $2$-flows in $G$ such that $supp(f_{1})\cup supp(f_{2}) = E(G)$, and $H_{i}$ is a spanning subgraph in $G$ with edge set $supp(f_{i})$, $i\in \{1,2\}$, and $I$ is a component of $H_{1}$.
      Thus, $\partial_{G}(I)\subseteq E(H_{2})$.
      Since $\partial_{G}(I)$ is an edge-cut of $G$, we have that $\partial_{G}(I)$ is also an edge-cut of $H_{2}$.
      Note that $H_{2}$ is an even graph.
      Therefore, $\left| \partial_{G}(I)\right|$ is even.
      Thus, all edges of $\partial_{G}(I)$ contribute an even number of vertices to $I_{2}$.
      Hence, $I_{2}$ is an even circuit.

      The process of converting $(I, \sigma)$ into $(I_{2}, \sigma_{2})$ is called $3$-regularizing of $I$, and we call $I_{2}$ the $2$-normal graph of $I$.
      Let $(G', \sigma')$ be the signed graph obtained from $(G, \sigma)$ by $3$-regularizing all components of $H_{1}$, and let $(J, \sigma')$ be the union of the $2$-normal graphs of all components of $H_{1}$.
      It is easy to see that $(G',\sigma')$ is a flow-admissible signed cubic graph and $J$ is a $2$-factor of $G'$.
      Since every component of $J$ is an even circuit, $(G',\sigma')$ is a flow-admissible signed $3$-edge-colorable cubic graph.
      Thus, Statement {\rm (1)} holds.

      Each component in $H_{1}$ has a unique $2$-normal graph in $J$, which is a component in $J$.
      Conversely, every component in $J$ is obtained from a component in $H_{1}$ by $3$-regularizing.
      Therefore, there exists a natural bijection $f:C( H_{1})\rightarrow C(J)$ such that for any $I \in C(H_{1})$, $f(I)$ is a component of $J$ obtained from $I$ by $3$-regularizing.
      Since every edge we added is positive, we have $\sigma(I)=\sigma'(f(I))$.
      Consequently, Statement (2) holds.

      It is easy to see that $S=E(G')\setminus E(G)$ is the set of all the edges we added.
      Thus, for any edge $e\in S$, we have $\sigma'(e)=+1$.
      In order to show the structure of $(G'/S,\sigma'\mid _{E(G'/S)})$, we will show that there is a decomposition $\{E(S_{u}):u\in V(G)\}$ of $S$, where $(S_{u},\sigma')$ is an induced all-positive subgraph in $(G',\sigma')$.
      Next, we introduce the vertex set of the graph $S_{u}$.

      Let $u$ be a vertex of $G$ that is in a component $I'$ of $H_{1}$.
      In the process of 3-regularizing $I'$, the vertex $ u $ is initially split into $\frac{d_{I'}(u)}{2}$ vertices, denoted by $ u_0, u_1, u_2, \ldots, u_{\frac{d_{I'}(u)}{2}-1} $, where $ u_0 = u $.
      Subsequently, for any $i\neq 0$, each vertex $u_{i}$ is blown up into $\frac{d_{I'}(u)}{2}+1$ vertices $u_{i}^{1}, u_{i}^{2}, \ldots, u_{i}^{\frac{d_{I'}(u)}{2}+1}$ if $d_{I'}(u)\neq4$, and each vertex $u_{i}$ is blown up into $\frac{d_{I'}(u)}{2}+2=4$ vertices $u_{i}^{1}, u_{i}^{2}, u_{i}^{3}, u_{i}^{4}$ if $d_{I'}(u)=4$.
      For $i=0$, vertex $u_{0}$  is blown up into $d_{G}(u)-\frac{d_{I'}(u)}{2}+1$ vertices $u_{0}^{1}, u_{0}^{2}, \ldots, u_{0}^{d_{G}(u)-\frac{d_{I'}(u)}{2}+1}$ if $d_{I'}(u)\neq4$, and vertex $u_{0}$ is blown up into $d_{G}(u)-\frac{d_{I'}(u)}{2}+2=d_{G}(u)$ vertices $u_{0}^{1}, u_{0}^{2}, \ldots, u_{0}^{d_{G}(u)}$ if $d_{I'}(u)=4$.
      Let $B(u_{i})$ be the set of vertices blown up from $u_{i}$, $i\in [0,\frac{d_{I'}(u)}{2}-1]$.
      Therefore, the vertex set $V(S_{u})=\bigcup_{i\in[0,\frac{d_{I'}(u)}{2}-1]}B(u_{i})$.

      Note that every edge of $S_{u}$ is an element in $S$.
      Thus, $(S_{u},\sigma')$ is all-positive.
      Conversely, for any edge $e \in S$, the edge $e$ is an element in some $E(S_{u})$, where $u\in V(G)$.
      Thus, $\bigcup_{u\in V(G)}E(S_{u})=S$.
      It is evident that and $V(S_{y})\cap V(S_{z})=\emptyset$ if $y$ and $z$ are distinct vertices of $G$.
      Therefore, $\{E(S_{u}):u\in V(G)\}$ is a decomposition of $S$.
      Furthermore, $\partial_{G'}(S_{u})=\partial_{G}(u)$, for any vertex $u\in V(G)$.

      Let $u'$ be the vertex in the graph $G'/S$ obtained by contracting the edges in the set $ E(S_u) $.
      Define a mapping $g: G'/S\rightarrow G$ such that $g(u')=u$.
      Consider $a$ and $b$ as two distinct vertices of $G$.
      Then $E_{G'}[V(S_{a}),V(S_{b})]=E_{G}[\{a\},\{b\})]$.
      Thus, $g$ establishes an isomorphism between $G'/S$ and $G$.

      If $G$ has no multiple edges, then $g$ also acts as an isomorphism between $(G'/S,\sigma'\mid _{E(G'/S)})$ and $(G,\sigma)$, where $\sigma'\mid _{E(G'/S)}$ denotes the restriction of $\sigma'$ to $E(G'/S)$.
      Note that $E(G'/S)=E(G)$.
      In cases where $G$ contains multiple edges, let $g(e)=e$.
      Then $g$ remains an isomorphism between $(G'/S,\sigma'\mid _{E(G'/S)})$ and $(G,\sigma)$.
      Therefore, $(G'/S,\sigma'\mid _{E(G'/S)})\cong(G,\sigma)$, validating Statement (3).
    \end{proof}

    \begin{rmk}
    We note that a similar reduction method was introduced in \cite{LMSZ25}.
     However, our reduction distinguishes itself by explicitly transforming Eulerian subgraphs in 4-NZF-admissible graphs into bichromatic circuits in the resulting 3-regular graph, while crucially preserving the sign of these Eulerian subgraphs throughout the transformation.
    \end{rmk}

    Let $ f $ be a $ k $-NZF of the signed graph $(G, \sigma)$, and let $ S $ be a set of positive edges in $(G, \sigma)$. Consider the signed graph $(G/S, \sigma|_{G/S})$, which is obtained by contracting all edges in $ S $.
    For simplicity, we denote it by $(G/S, \sigma)$.
    After contracting the edges in $ S $, there exists a $ k $-NZF, denoted by $ f|_{G/S} $, in $(G/S, \sigma)$. Here, $ f|_{G/S} $ represents the restriction of $ f $ to the edge set $ E(G/S) $.
    Let us recall Theorem \ref{equi-relation}.

    \begin{thm}\label{equi-relation}
      Let $k$ be a positive integer. Then the following statements are equivalent:

    {\rm(1)} Every flow-admissible signed $4$-NZF-admissible graph admits a nowhere-zero $k$-flow;

    {\rm(2)} Every flow-admissible signed $3$-edge-colorable cubic graph admits a nowhere-zero $k$-flow.
    \end{thm}

    \begin{proof}
      It is straightforward that ${\rm(1)}$ implies ${\rm(2)}$ since every $3$-edge-colorable cubic graph is $4$-NZF-admissible.
      Therefore, we only need to prove that ${\rm(2)}$ implies ${\rm(1)}$.

      Let $(G,\sigma)$ be a flow-admissible signed $4$-NZF-admissible graph.
      According to Lemma \ref{lem-4-into-3}, there exists a flow-admissible signed $3$-edge-colorable cubic graph $(G',\sigma')$ such that $(G'/S,\sigma')\cong(G,\sigma)$, where $S$ is a set of positive edges.
      Since every flow-admissible signed $3$-edge-colorable cubic graph admits a $k$-NZF, $(G',\sigma')$ admits a $k$-NZF.
      Consequently, $(G'/S, \sigma')$ admits a $ k $-NZF, and therefore, so does $(G, \sigma)$.
    \end{proof}





    Let $G_{1}, G_{2}, \cdots, G_{t}$ be subgraphs of $G$.
    The notation $ G_{1} \bigtriangleup G_{2} \bigtriangleup \cdots \bigtriangleup G_{t} $ represents the symmetric difference of these subgraphs.
    The following theorem shows the equivalence in the admission of $k$-NZF among signed supereulerian graphs, signed Hamiltonian graphs and signed Hamiltonian cubic graphs.

    \begin{thm}\label{equi-s-h-hc}
       Let $k$ be a positive integer. Then the following statements are equivalent:

      {\rm(1)} Every flow-admissible signed supereulerian graph admits a nowhere-zero $k$-flow;

      {\rm(2)} Every flow-admissible signed Hamiltonian graph admits a nowhere-zero $k$-flow;

      {\rm(3)} Every flow-admissible signed Hamiltonian cubic graph admits a nowhere-zero $k$-flow.
    \end{thm}

    \begin{proof}
      It is trivial that {\rm(1)} implies {\rm(2)} and {\rm(2)} implies {\rm(3)}.
      Thus, we only need to prove that {\rm(3)} implies {\rm(1)}.

      Since $G$ is a supereulerian graph, it contains a spanning Eulerian subgraph $H_{1}$.
      For any edge $ e \in E(G) \setminus E(H_{1}) $, there exists a circuit in $ H_{1} \cup e $ that contains $ e $, we denote it by $ C_{e} $.
      Let $H_{2}$ $=\mathop{\bigtriangleup}_{e\in E(G)\setminus E(H_{1})} C_{e}$.
      Then $H_{2}$ is an even graph.
      Let $f_{1}$ be a $2$-flow with $supp(f_{1})=E(H_{1})$ and $f_{2}$ be a $2$-flow with $supp(f_{2})=E(H_{2})$.
      Therefore, $supp(f_{1})\cup supp(f_{2})=E(G)$.

      Let $(G,\sigma)$ be a flow-admissible signed supereulerian graph.
      By Lemma \ref{lem-4-into-3}, there exists a flow-admissible signed $3$-edge-colorable cubic graph $(G',\sigma')$ such that $(G'/S,\sigma')\cong(G,\sigma)$, where $S$ is a set of positive edges.
      And there exists a $2$-factor $J$ of $G'$ such that there is a bijection $f: C(H_{1}) \rightarrow C(J)$.
      Since $\left| C(H_{1})\right|=1$, we have $\left| C(J)\right| =1$.
      Thus, $J$ is a Hamiltonian circuit of $G'$.
      Therefore, $(G',\sigma')$ is a signed Hamiltonian cubic graph.
      Since every flow-admissible signed Hamiltonian cubic graph admits a $k$-NZF, $(G',\sigma')$ admits a $k$-NZF.
      Thus, $(G'/S,\sigma')$ admits a $k$-NZF, and so does $(G,\sigma)$.
    \end{proof}



%

    For $\mathbb{Z}_{k}$-flow, we can prove that a class of flow-admissible signed supereulerian graphs admits a $\mathbb{Z}_{4}$-NZF.
    To conclude this section with an application of Lemma \ref{lem-4-into-3}, we prove that every signed supereulerian graph with a spanning even Eulerian subgraph admits a $\mathbb{Z}_{4}$-NZF.
    Before proceeding, it is necessary to define some terms and introduce relevant lemmas.
    A signed graph $(G, \sigma)$ is called \textit{antibalanced} if all even circuits in $ G $ are balanced and all odd circuits are unbalanced.
    M\'{a}\v{c}ajov\'{a} et al. \cite{MS15} provide the following characterization of signed cubic graphs that admit a $\mathbb{Z}_{4}$-NZF.

    \begin{thm}\cite{MS15}\label{H-Z4-f}
      A signed cubic graph admits a $\mathbb{Z}_{4}$-NZF if and only if it has an antibalanced $2$-factor.
    \end{thm}

    Recall that an even Eulerian graph is a signed Eulerian graph with an even number of negative edges.
    For $\mathbb{Z}_{4}$-NZF, we have the following corollary.

    \begin{coro}\label{BHC-Z4-flow}
    Every signed supereulerian graph with a spanning even Eulerian subgraph  admits a $\mathbb{Z}_{4}$-NZF.
    \end{coro}

    \begin{proof}
      Let $(G,\sigma)$ be a signed supereulerian graph with a spanning even Eulerian subgraph $(H,\sigma)$.
      As mentioned in the proof of Theorem \ref{equi-s-h-hc}, there are two $2$-flows, $f_{1}$ and $f_{2}$, on $G$ such that $supp(f_{1}) \cup supp(f_{2}) = E(G)$ and $supp(f_{1}) = E(H)$.
      According to Lemma \ref{lem-4-into-3}, there exists a signed $3$-edge-colorable cubic graph $(G',\sigma')$ with a balanced Hamiltonian circuit such that $(G'/S,\sigma')\cong(G,\sigma)$, where $S$ is a set of positive edges.
      Note that an even circuit with positive sign is both balanced and antibalanced.
      Consequently, if a Hamiltonian circuit in a cubic graph is antibalanced, it is also balanced due to its even length.
      Thus, according to Theorem \ref{H-Z4-f}, every signed Hamiltonian cubic graph with a balanced Hamiltonian circuit admits a $\mathbb{Z}_{4}$-NZF.
      Therefore, $(G',\sigma')$ admits a $\mathbb{Z}_{4}$-NZF, and so does $(G,\sigma)$.
    \end{proof}

    \section{Nowhere-zero $6$-flows on signed supereulerian graphs with a spanning even Eulerian subgraph}

    In this section, we discuss the existence of a $6$-NZF on a signed supereulerian graph with a spanning even Eulerian subgraph, as follows.

    \begin{thm}\label{see-6-flow}
      Let $(G,\sigma)$ be flow-admissible.
      If $(G,\sigma)$ has a spanning even Eulerian subgraph, then $(G,\sigma)$ admits a nowhere-zero $6$-flow.
    \end{thm}

    The following lemma shows that Theorem \ref{see-6-flow} can be reduced to the problem of deciding whether a flow-admissible signed Hamiltonian graph with a balanced Hamiltonian circuit admits a $6$-NZF.

    \begin{lem}\label{equi}
      Let $k$ be a positive integer. Then the following statements are equivalent:

      {\rm(1)} Every flow-admissible signed supereulerian graph with a spanning even Eulerian subgraph admits a $k$-NZF;

      {\rm(2)} Every flow-admissible signed Hamiltonian graph with a balanced Hamiltonian circuit admits a $k$-NZF;

      {\rm(3)} Every flow-admissible signed Hamiltonian cubic graph with a balanced Hamiltonian circuit admits a $k$-NZF.
    \end{lem}
    \begin{proof}
      We only need to show that (3) implies (1).
      Let $(G,\sigma)$ be a flow-admissible signed supereulerian graph with a spanning even Eulerian subgraph $(H,\sigma)$.
      As mentioned in the proof of Theorem \ref{equi-s-h-hc}, there are two $2$-flows, $f_{1}$ and $f_{2}$, on $G$ such that $supp(f_{1}) \cup supp(f_{2}) = E(G)$ and $supp(f_{1}) = E(H)$.
      By Lemma \ref{lem-4-into-3}, there exists a flow-admissible signed Hamiltonian cubic graph $(G',\sigma')$ with a balanced Hamiltonian circuit such that $(G'/S,\sigma')\cong(G,\sigma)$, where $S$ is a set of positive edges.
      By Statement (3), $(G',\sigma')$ admits a $k$-NZF, so does $(G'/S,\sigma')$.
      Therefore, $(G,\sigma)$ admits a $k$-NZF.
    \end{proof}

    By Lemma \ref{equi}, in order to prove Theorem \ref{see-6-flow}, it therefore suffices to prove Theorem \ref{BH-6-flow}.

    \begin{thm}\label{BH-6-flow}
      Let $(G,\sigma)$ be flow-admissible.
      If $(G,\sigma)$ has a balanced Hamiltonian circuit, then $(G,\sigma)$ admits a nowhere-zero $6$-flow.
    \end{thm}

    Before we proceed, we require the following lemma.

    \begin{lem}\cite{CLLZ18}\label{Z2-3-flow}
      If a signed graph $(G,\sigma)$ is connected and admits a $\mathbb{Z}_{2}$-flow $f_{1}$ such that $supp(f_{1})$ has an even number of negative edges, then it also admits a $3$-flow $f_{2}$ with $supp(f_1) = \{e \in E(G): f_{2}(e) = \pm1\}$.
    \end{lem}

    By Lemma \ref{Z2-3-flow}, we have the following lemma that shows the existence of a $3$-flow in a signed graph with an all-positive Hamiltonian circuit.

    \begin{lem}\label{even-3-flow}
       Let $(G,\sigma)$ be a signed graph with an all-positive Hamiltonian circuit $(H,\sigma)$.
       If $\left| E_{N}(G,\sigma) \right|$ is even, then $(G,\sigma)$ admits a $3$-flow $f$ such that $E(G)\setminus E(H)\subseteq \{e \in E(G): f(e) = \pm1\}$.
    \end{lem}

    \begin{proof}
      For an edge $e \in H$, there exists a Hamiltonian path $H \setminus e$ of $G$, denoted by $P$.
      Then, for any edge $e_{1} \in E(G)\setminus E(H)$, we have $P\cup \{e_{1}\}$ forms a unique circuit, denoted by $C_{e_{1}}$.

      The symmetric difference $\mathop{\bigtriangleup }_{e\in E(G)\setminus E(H)} C_{e}$, denoted by $H'$, contains all edges of $E(G)\setminus E(H)$.
      Therefore, $(H',\sigma)$ is an even graph with an even number of negative edges.
      It is evident that every signed even graph admits a $\mathbb{Z}_{2}$-NZF.
      Consequently, $(H',\sigma)$ also admits a $\mathbb{Z}_{2}$-NZF $f'$.

      Since $G$ is connected, by Lemma \ref{Z2-3-flow}, $(G,\sigma)$ admits a $3$-flow $f$ such that $E(G)\setminus E(H)\subseteq\{e \in E(G): f(e) = \pm1\}$.
    \end{proof}

    The following theorem shows that if a signed graph with an all-positive Hamiltonian circuit has an even number of negative edges, then it admits a $6$-NZF.

    \begin{lem}\label{even-6-flow}
      Let $(G,\sigma)$ be a signed graph with an all-positive Hamiltonian circuit $(H,\sigma)$. If $\left| E_{N}(G,\sigma) \right|$ is even, then $(G,\sigma)$ admits a $6$-NZF.
    \end{lem}
    \begin{proof}
      According to Lemma \ref{even-3-flow}, $(G,\sigma)$ admits a $3$-flow $f_{1}$ such that $supp(f_{1})\supseteq E(G)\setminus E(H)$.
      It is important to note that $(G,\sigma)$ admits a $2$-flow $f_{2}$ with $supp(f_{2})=E(H)$, since $(H,\sigma)$ is all-positive.
      Therefore, $f_{1}+3f_{2}$ forms a $6$-NZF on $(G,\sigma)$.
    \end{proof}

    Let $H$ be a Hamiltonian circuit in the graph $G$, with the vertex sequence $v_{0}v_{1}\cdots v_{n-1}v_{0}$.
    Let $e_{1}$ and $e_{2}$ be two edges in $E(G)\setminus E(H)$.
    Suppose the ends of $e_{1}$ are $v_{i}$ and $v_{j}$, and the ends of $e_{2}$ are $v_{k}$ and $v_{l}$.
    We say that $e_{1}$ and $e_{2}$ are {\it intersect} along $H$ if $i<k<j<l$.
    Meanwhile, $e_{1}$ and $e_{2}$ are said to be {\it parallel} along $H$ if $k<i<j<l$ or $i<j<k<l$.
    These terms originate from plane geometry. When we draw the Hamiltonian circuit $H$ as a circle on a plane, and connect four distinct points with two line segments, these segments either intersect or do not intersect.

   The following lemma discusses the existence of a $6$-NZF in a signed graph that contains an all-positive Hamiltonian circuit and has two negative edges that intersecting along this circuit.

    \begin{lem}\label{odd-intersect-6-flow}
     Let $(G,\sigma)$ be a flow-admissible signed graph with an all-positive Hamiltonian circuit $(H,\sigma)$.
     If two negative edges intersect along $H$, then $(G,\sigma)$ admits a $6$-NZF.
    \end{lem}
    \begin{proof}
      If $\left| E_{N}(G,\sigma) \right|$ is even, then it is a direct corollary of Lemma \ref{even-6-flow}.
      Therefore, for the remainder of the proof, we assume that $\left| E_{N}(G,\sigma) \right|$ is odd.
      Let $e_{1}$ and $e_{2}$ be two negative edges that intersect along $H$.
      Then $(G\setminus e_{1},\sigma)$ is a signed graph with an all-positive Hamiltonian circuit, and $\left|E_{N}(G\setminus e_{1},\sigma) \right|$ is even.
      By Lemma \ref{even-3-flow}, $(G\setminus e_{1},\sigma)$ admits a $3$-flow $f_{1}$ and $E(G)\setminus (E(H)\cup \{e_{1}\})\subseteq supp(f_{1})$.
      Let $e_{1}=u_{1}v_{1}$ and $e_{2}=u_{2}v_{2}$.
      There exists a $3$-flow of $(H \cup \{e_{1}, e_{2}\}, \sigma)$, denoted by $f_{2}$, as illustrated in Fig. \ref{IH-3-flow} (omitting edges with a weight of $0$ and vertices of degree $2$).
      This $3$-flow is constructed as follows:
      \begin{equation}
      f_{2}(e) =
      \begin{cases}
       2, & \text{if } e \in \{e_{1}, e_{2}\}; \\
       1, & \text{if } e \in E(H). \nonumber
       \end{cases}
      \end{equation}

      Since $f_{1}(e_{2}) = \pm1$, $f_{2}(e_{2}) = 2$, and $f_{2}(E(H)) = \{1\}$, we conclude that either $2f_{1} + f_{2}$ or $2f_{1} - f_{2}$ is a $6$-NZF on $(G,\sigma)$.
    \end{proof}

    \begin{figure}[h]
    \centering
    \begin{minipage}[b]{.5\textwidth}
    \centering
    \begin{tikzpicture}[scale=0.8]

    \draw (0,0) circle (2);

    \draw[<-<] (-1,1.735) arc (120:150:2);
    \draw (-1.414,1.414)--(-1.414,1.414) node[midway,above]{1};
    \draw[<-<] (1,-1.735) arc (300:330:2);
    \draw (1.414,-1.414)--(1.414,-1.414) node[midway,below]{1};

    \draw[>->] (1.735,1) arc (30:60:2);
    \draw (1.414,1.414)--(1.414,1.414) node[midway,above]{1};
    \draw[>->] (-1.735,-1) arc (210:240:2);
    \draw (-1.414,-1.414)--(-1.414,-1.414) node[midway,below]{1};

    \draw [densely dashed](0,2)--(0,-2);
    \draw [densely dashed][>-<] (0,1)--(0,-1);
    \draw (0,1)--(0,1) node[midway,right]{2};
    \draw [densely dashed](-2,0)--(2,0);
    \draw [densely dashed][<->] (-1,0)--(1,0);
    \draw (-1,0)--(-1,0) node[midway,below]{2};

   \fill (0,2) circle (.1);
   \node at (0,2)[above]{$u_{2}$};
   \fill (-2,0) circle (.1);
   \node at (-2,0)[left]{$u_{1}$};
   \fill (2,0) circle (.1);
   \node at (2,0)[right]{$v_{1}$};
   \fill (0,-2) circle (.1);
   \node at (0,-2)[below]{$v_{2}$};
     \end{tikzpicture}
      \caption{A $3$-NZF $f_{2}$ on $(H\cup \{e_{1},e_{2}\},\sigma)$.}\label{IH-3-flow}
      \end{minipage}%
      \begin{minipage}[b]{.5\textwidth}
    \centering
    \begin{tikzpicture}[scale=0.8]

    \draw (0,0) circle (2);


    \draw[>->] (1.28,1.535) arc (50:-50:2);
    \draw (2,0)--(2,0) node[midway,right]{1};
    \draw[>->] (-1.28,-1.535) arc (230:130:2);
    \draw (-2,0)--(-2,0) node[midway,left]{1};

    \draw[>->] (0.7,1.88) arc (70:110:2);
    \draw (0,2)--(0,2) node[midway,above]{1};
    \draw[>->] (0.7,-1.88) arc (-70:-110:2);
    \draw (0,-2)--(0,-2) node[midway,above]{3};

    \draw [densely dashed](-1,1.75)--(-1,-1.75);
    \draw [densely dashed][>-<] (-1,0.7)--(-1,-0.7);
    \draw (-1,0)--(-1,0) node[midway,right]{2};

    \draw [densely dashed](1,1.75)--(1,-1.75);
    \draw [densely dashed][<->] (1,0.7)--(1,-0.7);
    \draw (1,0)--(1,0) node[midway,right]{2};

   \fill (-1,1.75) circle (.1);
   \node at (-1,1.75)[above]{$u_{1}$};
   \fill (-1,-1.75) circle (.1);
   \node at (-1,-1.75)[below]{$v_{1}$};
   \fill (1,1.75) circle (.1);
   \node at (1,1.75)[above]{$u_{2}$};
   \fill (1,-1.75) circle (.1);
   \node at (1,-1.75)[below]{$v_{2}$};

     \end{tikzpicture}
      \caption{A $4$-NZF $f_{1}$ on $(G_{1},\sigma)$.}\label{G2-4-flow}
      \end{minipage}%
   \end{figure}

    We present the proof of Theorem \ref{BH-6-flow} below.

    \begin{proof}[The proof of Theorem \ref{BH-6-flow}]

      We assume that the balanced Hamiltonian circuit $(H, \sigma)$ is all-positive; otherwise, we switch at some vertices of $H$ to ensure that every edge in $H$ is positive.
      By Lemma \ref{equi}, we assume that $(G,\sigma)$ is a signed cubic graph.
      Consequently, any two edges in $E(G)\setminus E(H)$ are either intersecting or parallel along $H$.
      By Lemma \ref{even-6-flow} and Lemma \ref{odd-intersect-6-flow}, we can further assume that $\left| E_{N}(G,\sigma)\right|$ is odd and that any two negative edges are parallel along $H$.

      Let $e^{1}$ and $e^{2}$ be two negative edges with ends $u_{1}$, $v_{1}$ and $u_{2}$, $v_{2}$, respectively.
      Consider the path $P= v_{1}e_{0}w_{1}e_{1}w_{2}e_{2}\cdots w_{k}e_{k}v_{2}$, where $e_{i}$ is an edge and $w_{i}$ is a vertex, $i\in [0,k]$.
      This path in $H$ that connects $v_{1}$ and $v_{2}$ and does not contain $u_{1}$ and $u_{2}$ as vertices.
      We may assume that every vertex in $V(P)\setminus \{v_{1},v_{2}\}$ is incident only with positive edges.
      Otherwise, we replace $e_{1}$ by the negative edge incident to some vertex $w_{i}$, where $i\in [1,k]$.
      Thus, these two negative edges $e_{1}$ ,$e_{2}$ and path $P$ can definitely be found in the signed graph $(G,\sigma)$.

      Define $G_{1}=H\cup \{e^{1},e^{2}\}$.
      It is easy to see that $(G_{1},\sigma)$ admits a $4$-NZF $f_{1}$ (see Fig. \ref{G2-4-flow}, we omit the edge which weighted by $0$, and the vertices of degree $2$).
      Next, we will construct a $3$-flow $f_{2}$ on $(G,\sigma)$ such that $f_{1}+ 2f_{2}$ or $f_{1}- 2f_{2}$ is a $6$-NZF on $(G,\sigma)$.

      Let $M_{P}=\{e \in E(G): e\ has\ at\ least\ one\ end\ in\ P\}\setminus \{e^{1},e^{2}\}$.
      Note that, $\sigma(e)=+1$ for all $e \in M_{P}$.
      Let $M= E(G)\setminus (E(H)\cup \{e^{2}\})$.
      We will remove certain edges from $(G, \sigma)$ to obtain a signed subgraph $(G_{2}, \sigma)$ of $(G, \sigma)$.
      Depending on the parity of $|E(P)|$, we will discuss the structure of $(G_{2}, \sigma)$.
      There are two distinct cases to consider.\\

      \textbf{Case 1.} $\left| E(P) \right|$ is odd.

      After removing all edges of $\{ e^{2}, e_{0}, e_{2}, \cdots, e_{k}\}$ from $(G, \sigma)$, denote the resulting signed graph by $(G_{2}, \sigma)$.
      Since $\{e_{0}, e_{2}, \cdots, e_{k}\}$ is a matching in $G$, each vertex in $V(P) \setminus \{v_{1}, v_{2}\}$ has degree $2$ in $G_{2}$.
      Given that $M_{P}$ and $\{e_{1}, e_{3}, \cdots, e_{k-1}\}$ are disjoint matchings, $M_{P} \cup \{e_{1}, e_{3}, \cdots, e_{k-1}\}$ induces a disjoint union of paths and circuits, denoted by $P_{1}$, $P_{2}$, $\cdots$, $P_{x}$ and $C_{1}$, $C_{2}$, $\cdots$, $C_{y}$, respectively.
      Note that each circuit $C_{i}$ is a component of $G_{2}$, and $(C_{i}, \sigma)$ is all-positive for each $i \in [1, y]$.
      Let $P'$ be a path induced by the edge set $E(H) \setminus E(P)$, and let $M' = M \setminus M_{P}$.
      Note that in $G_{2}$, there is a single vertex $v_{2}$ of degree $1$, and no edge in $M'$ has $v_{2}$ as an end.
      Additionally, there is no path $P_{j}$ containing $v_{2}$ as an end, for $j \in [1, x]$.
      Therefore, the ends of every $P_{j}$ and each edge in $M'$ are vertices of degree $3$ in $G_{2}$.
      As a result, all edges in $M'$ have their ends in $V(P') \setminus \{v_{2}\}$, and likewise, each $P_{j}$ has its ends in $V(P') \setminus \{v_{2}\}$, for $j \in [1, x]$.
      Thus, $G_{2}$ has $y+1$ components.\\

      \textbf{Case 2.} $\left| E(P) \right|$ is even.

      Let $e^*$ be an edge incident with $v_{1}$ that is different from $e^{1}$ and $e_{0}$.
      Note that $e^*\in E(H)$.
      By removing the edges in the set $\{e^{*}, e^{2}, e_{1}, e_{3}, \cdots, e_{k}\}$ from $(G,\sigma)$, the resulting signed graph is denoted by $(G_{2},\sigma)$.
      Since the set $\{e^{*}, e_{1}, e_{3}, \cdots, e_{k}\}$ forms a matching in $G$, every vertex of $V(P)\setminus \{v_{2}\}$ has degree $2$ in $G_{2}$.
      Since $M_{P}\cup \{e^{1}\}$ and $\{e_{0},e_{2},\cdots,e_{k-1}\}$ are disjoint matchings, the union $M_{P}\cup \{e^{1},e_{0},e_{2},\cdots,e_{k-1}\}$ induces a disjoint union of paths and circuits.
      These are denoted by $P_{1}$, $P_{2}$, $\cdots$, $P_{x}$ and $C_{1}$, $C_{2}$, $\cdots$, $C_{y}$, respectively.
      It is easy to see that each circuit $C_{i}$ is a component of $G_{2}$.
      Since $e^{1}$ has only one end in $P$, it cannot be present in any circuit $C_{i}$, for $i \in [1,y]$.
      Thus, each $(C_{i},\sigma)$ is all-positive and $e^{1} \in E(P_{j})$ for a unique $j \in [1, y]$.
      Let $P'$ be a path induced by the edge set $E(H) \setminus (E(P)\cup \{e^{*}\})$ and let $M' = M \setminus (M_{P}\cup \{e^{1}\})$.
      Note that in $G_{2}$, there is only one vertex $v_{2}$ has degree $1$, and no edge in $M'$ containing $v_{2}$ as an end.
      Additionally, no path $P_{j}$ terminates at $v_{2}$ for $j \in [1, x]$.
      Therefore, the ends of every $P_{j}$ and every edge in $M'$ must be vertices of degree $3$ in $G_{2}$.
      Consequently, the ends of all edges in $M'$ are in $V(P')\setminus \{v_{2}\}$, and the ends of each $P_{j}$ are also within $V(P')\setminus \{v_{2}\}$ for $j \in [1, x]$.
      Thus, $G_{2}$ has $y+1$ components.\\

      Let $G_{3}$ be the component of $G_{2}$ that contains the $P'$.
      Define the set $$\mathcal{C}=\{C:\ C\ is\ the\ unique\ circuit\ in\ P'\cup e\ or\ P'\cup P_{j},\ j\in [1,x]\}.$$
      The symmetric difference $\mathop{\bigtriangleup }_{C\in \mathcal{C}} C$, denoted by $G_{4}$, contains all edges in $E(G_{3}) \setminus E(P')=M' \cup (\bigcup_{j\in [1,y]} E(P_{j}))=M\cup \{e_{1}, e_{3}, \cdots, e_{k-1}\}$.
      Let $M^{*}=M\cup \{e_{1}, e_{3}, \cdots, e_{k-1}\}$.
      Since $G_{4}$ is an even graph, $(G_{4}, \sigma)$ is a signed even graph.
      Every signed even graph admits a $\mathbb{Z}_{2}$-NZF, even if it is not necessarily flow-admissible.
      Thus, $(G_{4}, \sigma)$ admits a $\mathbb{Z}_{2}$-NZF $f_{4}$, and $M^{*} \subseteq \text{supp}(f_{4})$.
      Since $e^{2} \notin E(G_{4})$ and $M\subseteq E(G_{4})$, the signed graph $(G_{4}, \sigma)$ has an even number of negative edges.
      This means that $supp (f_{4})$ has an even number of negative edges.
      Given that $G_{3}$ is connected and $f_{4}$ is a $\mathbb{Z}_{2}$-flow on $(G_{3},\sigma)$, the signed graph $(G_{3},\sigma)$ admits a $3$-flow $f_{3}$ such that $M^{*}\subseteq \{e \in E(G_{4}): f_{3}(e)=\pm1 \}$.
      Since each $(C_{i},\sigma)$ is all-positive, $(C_{i},\sigma)$ admits a $2$-NZF $g_{i}$ for each $i$.
      By combining these, $f_{2}=f_{3}+\mathop{\sum}_{i \in [1,y]}g_{i}$ forms a $3$-flow on $(G_{2},\sigma)$.
      We can verify that $f_{1}+ 2f_{2}$ or $f_{1}- 2f_{2}$ is a $6$-NZF on $(G,\sigma)$.

      For $E(P)$, we have $f_{1}(E(P))=\{+3\}$ and $f_{2}(E(P))\subseteq\{0,+1,-1\}$.
      Therefore, $(f_{1}\pm 2f_{2})(E(P))\subseteq\{3,+5,-5\}$.

      For $E(H)\setminus E(P)$, we have $f_{1}(E(H)\setminus E(P))=\{+1\}$ and $f_{2}(E(H)\setminus E(P))\subseteq\{0,+1,-1,+2,-2\}$.
      Therefore, $(f_{1}\pm 2f_{2})(E(H)\setminus E(P))\subseteq\{+1,-1,+3,-3,+5\}$.

      For $M\setminus \{e_{1}\}$, we have $f_{1}(M\setminus\{e_{1}\})=\{0\}$ and $f_{2}(M\setminus\{e_{1}\})\subseteq\{+1,-1\}$.
      Therefore, $(f_{1}\pm 2f_{2})(M\setminus\{e_{1}\})\subseteq\{+2,-2\}$.

      For $e_{1}$, we have $f_{1}(e_{1})=+1$ and $f_{2}(e_{1})\in\{+1,-1\}$.
      Therefore, $(f_{1}\pm 2f_{2})(e_{1})\in\{-1,3\}$.

      For $e_{2}$, we have $f_{1}(e_{2})=+1$ and $f_{2}(e_{2})=0$.
      Therefore, $(f_{1}\pm 2f_{2})(e_{2})=+1$.

      Thus, $f_{1}+ 2f_{2}$ or $f_{1}- 2f_{2}$ is a $6$-NZF on $(G,\sigma)$.
    \end{proof}

    A {\it Kotzig graph} is a cubic graph that has three $1$-factors such that the union of any two of them induces a Hamiltonian circuit.
    Schubert et al. \cite{SS15} prove that every flow-admissible signed Kotzig graph admits a $6$-NZF, i.e., Theorem \ref{k-6-flow}.
    According to Theorem \ref{BH-6-flow}, we provide an alternative proof of Theorem \ref{k-6-flow} as follows.

    \begin{thm}\cite{SS15}\label{k-6-flow}
      Let $(G, \sigma )$ be a flow-admissible signed cubic graph. If $G$ is a Kotzig graph, then $(G, \sigma )$ admits a $6$-NZF.
    \end{thm}

    \begin{proof}
      Let $F_{1}$, $F_{2}$ and $F_{3}$ be the three $1$-factors of $G$ such that the union of any two of them induces a Hamiltonian circuit.
      By the Pigeonhole Principle, there exist distinct indices $i, j \in \{1, 2, 3\}$ such that $\left| E_{N}(F_{i}, \sigma) \right| \equiv \left| E_{N}(F_{j}, \sigma) \right| \pmod{2}$.
      Thus, $(F_{i}\cup F_{j},\sigma)$ is a balanced Hamiltonian circuit of $(G, \sigma )$.
      By Theorem \ref{BH-6-flow}, $(G, \sigma )$ admits a $6$-NZF.
    \end{proof}

    \section{Nowhere-zero flows on signed abelian Cayley graphs}
    In this section, the nowhere-zero flows on signed abelian Cayley graphs are studied. All groups considered in this paper are finite.

    \subsection{Nowhere-zero $6$-flows on signed abelian Cayley graphs}
    \

    In this subsection, it is shown that every flow-admissible signed abelian Cayley graph admits a $6$-NZF.

    Let $X$ be a group and let $S$ be a subset of $X$ that is closed under taking inverses and does not contain the identity.
    The Cayley graph $Cay(X, S)$ is defined with vertex set $X$, where two vertices $g$ and $h$ are adjacent if and only if $ hg^{-1} \in S$.
    A Cayley graph $Cay(X,S)$ is said to be abelian if $X$ is abelian.

    A graph in which every vertex has equal degree $k$ is called regular of valency $k$.
   Because connected abelian Cayley graphs possess Hamiltonian circuits, they are supereulerian, and thus admit a $4$-NZF.
    Moreover, Poto\v{c}nik et al. \cite{PSS05} and N\'{a}n\'{a}siov\'{a} et al. \cite{NS09} showed that every abelian Cayley graph of valency at least $5$ admits a $3$-NZF.

    The main result of this subsection shows that such a class of flow-admissible signed $4$-NZF-admissible graphs admit a $6$-NZF, as follows.

    \begin{thm}\label{AC-6-flow}
      Every flow-admissible signed abelian Cayley graph admits a nowhere-zero $6$-flow.
    \end{thm}

    To prove Theorem \ref{AC-6-flow}, we introduce two fundamental structures in abelian Cayley graphs: the circular ladder and the M\"{o}bius ladder.
    Let $n\geq 1$ be an integer.
    A cubic graph is called a {\it circular ladder} if it is isomorphic to $C_{n} \Box K_{2}$, denoted by $CL_{n}$. (For the definition of the Cartesian product of graphs, see \cite{BM08} p. 30.)
    Let $V(CL_{n})=\{x_{0},x_{1},\cdots, x_{n-1},y_{0},y_{1},\cdots, y_{n-1}\}$, and $E(CL_{n})=\{x_{i}y_{i}: i\in \mathbb{Z}_{n} \}\cup\{x_{i}x_{i+1}: i\in \mathbb{Z}_{n}\}\cup\{y_{i}y_{i+1}: i\in \mathbb{Z}_{n}\}$.
    A cubic graph is defined as a {\it M\"{o}bius ladder} if it can be obtained from $CL_{n}$ by removing edges $x_{n-1}x_{0}$ and $y_{n-1}y_{0}$, and adding edges $x_{n-1}y_{0}$ and $x_{0}y_{n-1}$.
    This graph is denoted by $ML_{n}$.

    The following lemma shows that every connected cubic abelian Cayley graph is isomorphic to either $CL_{n}$ or $ML_{n}$.
    An element $a$ of a group $X$ is an involution if $a^{2} = 1_{e}$, where $1_{e}$ is the identity element of $X$.
    Specifically, $a$ is a central involution if it is an involution and commutes with every element $b \in X$, i.e., $ab = ba$ for all $b \in X$.
    If $Cay(X,S)$ is cubic, then $S$ includes a involution of $X$, because $\left| S \right|=3$ and it is closed under taking inverses.
    Furthermore, if $X$ is abelian, then $S$ includes a central involution of $X$.

    \begin{lem}\cite{NS09}
      Let $Cay(X,S)$ be a connected cubic Cayley graph such that $S$ contains a central involution of $X$. Then $Cay(X,S)$ is isomorphic to $CL_{n}$ or $ML_{n}$.
    \end{lem}

    Thus, a connected cubic abelian Cayley graph is isomorphic to either a circular ladder or a M\"{o}bius ladder.
    The following theorem shows that, to prove Theorem \ref{AC-6-flow}, it suffices to show that every flow-admissible $(CL_{n},\sigma)$ and $(ML_{n},\tau)$ admits a $6$-NZF.
    Note that every component of $Cay(X, S)$ is $\left| S \right|$-edge-connected because $Cay(X, S)$ is vertex-transitive.
    Raspaud et al. \cite{RZ11} showed that every flow-admissible signed $4$-edge-connected graph admits a $4$-NZF, as follows.

    \begin{thm}\cite{RZ11}\label{4EN-6-flow}
      Let $G$ be a $4$-edge-connected graph.
      If $(G,\sigma)$ is flow-admissible, then $(G,\sigma)$ admits $4$-NZF.
    \end{thm}

    Although we have not yet proven that every flow-admissible $(CL_{n},\sigma)$ and $(ML_{n},\tau)$ admits a $6$-NZF, we present a proof of Theorem \ref{AC-6-flow} here.

    \begin{proof}[The proof of Theorem \ref{AC-6-flow}]

      Let $\Gamma=Cay(X,S)$ be an abelian Cayley graph.
      Note that if $\Gamma=Cay(X,S)$ is not connected, then each component of $\Gamma=Cay(X,S)$ is isomorphic to an abelian Cayley graph $\Gamma_{1}=Cay(X_{1},S)$, where $X_{1}$ is a proper subgroup of $X$ and generated by $S$.

      Let $(\Gamma,\sigma)$ be flow-admissible.
      Since $(\Gamma,\sigma)$ is flow-admissible if and only if each component of $(\Gamma,\sigma)$ is flow-admissible, we can assume that each component of $(\Gamma,\sigma)$ is flow-admissible.
      Furthermore, if $\Gamma$ is not connected, then each component of $(\Gamma,\sigma)$ is isomorphic to a flow-admissible signed abelian Cayley graph.
      Thus, without loss of generality, we assume that $\Gamma$ is connected.

      If $\left|S\right|\geq 4$, then $\Gamma$ is $4$-edge-connected.
      By Theorem \ref{4EN-6-flow}, $(\Gamma,\sigma)$ admits a $4$-NZF.

      When $\left|S\right|= 3$, $\Gamma$ is isomorphic to either $CL_{n}$ or $ML_{n}$.
      By Theorems \ref{M-6} and \ref{C-6}, $(\Gamma,\sigma)$ admits a $6$-NZF.

      When $\left|S\right|= 2$, the signed graph $(\Gamma,\sigma)$ is a balanced circuit since $(\Gamma,\sigma)$ is flow-admissible.
      Consequently, there is a $2$-NZF in $(\Gamma,\sigma)$.

      For $\left|S\right|= 1$, $(\Gamma,\sigma)$ is not flow-admissible, leading to a contradiction.
    \end{proof}

    Let $G$ be a circular ladder or M\"{o}bius ladder.
    In the remainder of this subsection, we will prove that every flow-admissible $(G,\sigma)$ admits a $6$-NZF, as stated in Theorem \ref{M-6} and Theorem \ref{C-6}.
    In most cases, we can find a balanced Hamiltonian circuit in $(G,\sigma)$, and we usually assume that the balanced Hamiltonian circuit is all-positive due to the switching operation.
    Note that $(G,\sigma)$ is not flow-admissible if $\left| E_{N}(G,\sigma) \right| = 1$.

    The following theorem shows that every flow-admissible signed M\"{o}bius ladder admits a $6$-NZF.

    \begin{thm}\label{M-6}
      Every flow-admissible $(ML_{n},\sigma)$ admits a $6$-NZF.
    \end{thm}
    \begin{proof}
      By Theorem \ref{BH-6-flow}, it suffices to prove that there is a balanced Hamiltonian circuit in $(ML_{n},\sigma)$.
      The edge set $E(ML_{n}) \setminus \{e \in E(ML_{n}) : e=x_{i}y_{i}, i\in \mathbb{Z}_{n} \}$ induces a Hamiltonian circuit in $ML_{n}$, denoted by $C$.
      If $(C,\sigma)$ is balanced, then we are done.
      Therefore, for the remainder of the proof, we assume that $(C,\sigma)$ is unbalanced.
      We may assume that $(C,\sigma)$ has precisely one negative edge, $x_{0}y_{n-1}$.
      Otherwise we switch at some vertices of $C$ such that $x_{0}y_{n-1}$ is negative and other edges are positive.

      If $\sigma(x_{0}y_{0})=\sigma(x_{n-1}y_{n-1})$, then the sequence $x_{0}x_{1}x_{2}\cdots x_{n-1}y_{n-1}y_{n-2}\cdots y_{0}x_{0}$ forms a balanced Hamiltonian circuit.

      Suppose now that $\sigma(x_{0}y_{0})\neq\sigma(x_{n-1}y_{n-1})$.
      Without loss of generality, let $x_{n-1}y_{n-1}$ be negative.
      We claim that there exists another negative edge in $\{x_{i}y_{i}:i\in [1,n-2] \}$.
      Suppose, to the contrary, that $\{x_{i}y_{i}:i\in [1,n-2] \}$ contains no negative edges.
      Then we switch at $y_{n-1}$.
      The resulting signed graph has only one negative edge, $y_{n-2}y_{n-1}$, which contradicts the fact that $(ML_{n},\sigma)$ is flow-admissible.

      Meanwhile, we claim that there is another positive edge in $\{x_{i}y_{i}:i\in [1,n-2] \}$.
      Suppose, to the contrary, that $\{x_{i}y_{i}:i\in [1,n-2] \}$ contains no positive edges.
      Then we switch at $\{y_{n-1}, y_{n-2}\cdots, y_{2},y_{1}\}$.
      The resulting signed graph has only one negative edge,  $y_{0}y_{1}$, which leads to a contradiction.

      Hence, there exists a pair $(j,j+1)$, where $j\in[1,n-2]$, such that $\sigma(x_{j}y_{j}) \neq \sigma(x_{j+1}y_{j+1})$.
      Without loss of generality, we assume that $\sigma(x_{j}y_{j})=-1$.
      Then we switch at $\{x_{0},x_{1}, x_{2},\cdots, x_{j}\}$.
      The resulting signed graph is denoted by $(ML_{n},\sigma')$.
      In this resulting signed graph, the subgraph $(C,\sigma')$ has only one negative edge, $x_{j}x_{j+1}$.
      Furthermore, both $x_{j}y_{j}$ and $x_{j+1}y_{j+1}$ are positive in $(ML_{n},\sigma')$.
      Thus, $(C\setminus\{x_{j}x_{j+1},y_{j}y_{j+1}\})\cup \{x_{j}y_{j},x_{j+1}y_{j+1}\}$ forms an all-positive Hamiltonian circuit $$x_{j}y_{j}y_{j-1}\cdots y_{0}x_{n-1}x_{n-2}\cdots x_{j+1}y_{j+1}y_{j+2}\cdots y_{n-1}x_{0}x_{1}\cdots x_{j}.$$
      Therefore, there is a balanced Hamiltonian circuit in $(ML_{n},\sigma)$.
    \end{proof}

    In the remainder of this subsection, we will prove that every flow-admissible $(CL_{n},\sigma)$ admits a $6$-NZF.
    Before we proceed, we need to introduce some notation and terminology.
    Let $C_{x}=x_{0}x_{1}\cdots x_{n-1}x_{0}$ and $C_{y}=y_{0}y_{1}\cdots y_{n-1}y_{0}$ be two circuits of $CL_{n}$, and let $M$ be the $1$-factor of $CL_{n}$ with edge set $\{x_{i}y_{i} : i\in [0,n-1]\}$.
    These three subgraphs are edge disjoint, and $CL_{n}=C_{x}\cup C_{y}\cup M$.

    For $(CL_{n},\sigma)$ with $\sigma(x_{i}x_{i+1})=\sigma(y_{i}y_{i+1})=+1$, where the indices $ i $ and $ i+1 $ are taken modulo $ n $, and $n\geq 3$.
    An {\it $(m,i)$-extender of $(CL_{n},\sigma)$} is a signed graph obtained from $(CL_{n},\sigma)$ by replacing $x_{i}x_{i+1}$ and $y_{i}y_{i+1}$ by two all-positive paths of length $m+1$, denoted by $P_{x_{i}}=x_{i}x^{i}_{1}x^{i}_{2}\cdots x^{i}_{m}x_{i+1}$ and $P_{y_{i}}=y_{i}y^{i}_{1}y^{i}_{2}\cdots y^{i}_{m}y_{i+1}$, respectively, and adding edges $x^{i}_{j}y^{i}_{j}$ for $j\in [1,m]$, where $x^{i}_{j}y^{i}_{j}$ is negative if $j$ is odd, and positive if $j$ is even.
    There is an example, as shown in Fig. \ref{extending}.
    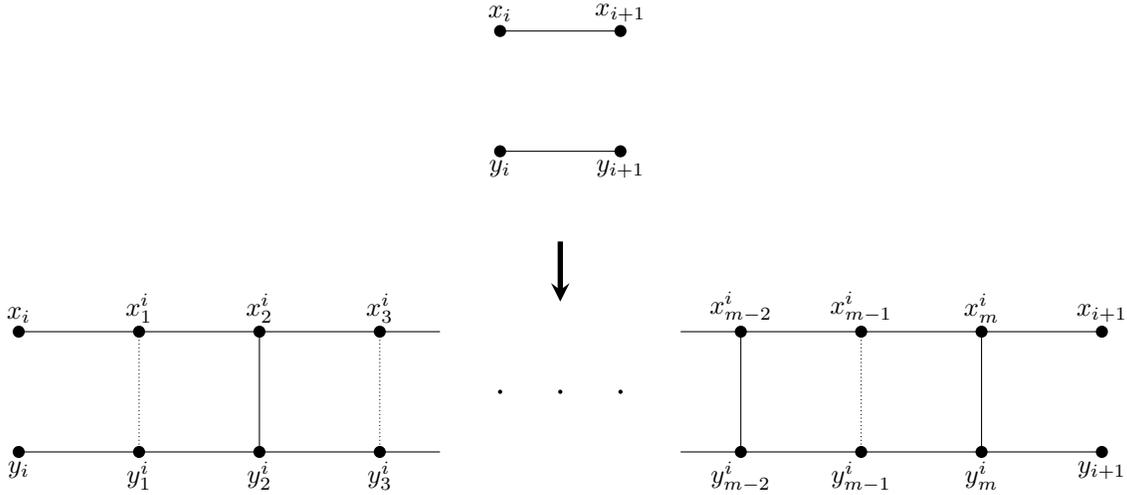
\begin{figure}[h]
    \centering
    \begin{tikzpicture}[scale=0.8]

        \draw (-1,4) -- (1,4);
        \draw (-1,2) -- (1,2);

        \draw (-9,-1) -- (-2,-1);
        \draw (-9,-3) -- (-2,-3);
        \draw (9,-1) -- (2,-1);
        \draw (9,-3) -- (2,-3);

        \draw [densely dotted](-7,-3) -- (-7,-1);
        \draw (-5,-3) -- (-5,-1);
        \draw [densely dotted](-3,-3) -- (-3,-1);
        \draw (7,-3) -- (7,-1);
        \draw [densely dotted](5,-3) -- (5,-1);
        \draw (3,-3) -- (3,-1);

      \fill (-1,4) circle (.1);
        \node at (-1,4)[above]{$x_{i}$};
        \fill (1,4) circle (.1);
        \node at (1,4)[above]{$x_{i+1}$};

        \fill (-1,2) circle (.1);
        \node at (-1,2)[below]{$y_{i}$};
        \fill (1,2) circle (.1);
        \node at (1,2)[below]{$y_{i+1}$};

        \fill (-9,-1) circle (.1);
        \node at (-9,-1)[above]{$x_{i}$};
        \fill (-9,-3) circle (.1);
        \node at (-9,-3)[below]{$y_{i}$};
        \fill (-7,-1) circle (.1);
        \node at (-7,-1)[above]{$x^{i}_{1}$};
        \fill (-7,-3) circle (.1);
        \node at (-7,-3)[below]{$y^{i}_{1}$};
        \fill (-5,-1) circle (.1);
        \node at (-5,-1)[above]{$x^{i}_{2}$};
        \fill (-5,-3) circle (.1);
        \node at (-5,-3)[below]{$y^{i}_{2}$};
        \fill (-3,-1) circle (.1);
        \node at (-3,-1)[above]{$x^{i}_{3}$};
        \fill (-3,-3) circle (.1);
        \node at (-3,-3)[below]{$y^{i}_{3}$};
        \fill (9,-1) circle (.1);
        \node at (9,-1)[above]{$x_{i+1}$};
        \fill (9,-3) circle (.1);
        \node at (9,-3)[below]{$y_{i+1}$};
        \fill (7,-1) circle (.1);
        \node at (7,-1)[above]{$x^{i}_{m}$};
        \fill (7,-3) circle (.1);
        \node at (7,-3)[below]{$y^{i}_{m}$};
        \fill (5,-1) circle (.1);
        \node at (5,-1)[above]{$x^{i}_{m-1}$};
        \fill (5,-3) circle (.1);
        \node at (5,-3)[below]{$y^{i}_{m-1}$};
        \fill (3,-1) circle (.1);
        \node at (3,-1)[above]{$x^{i}_{m-2}$};
        \fill (3,-3) circle (.1);
        \node at (3,-3)[below]{$y^{i}_{m-2}$};

        \fill (-1,-2) circle (1pt);
        \fill (0,-2) circle (1pt);
        \fill (1,-2) circle (1pt);


         \draw[-stealth,line width=2pt,black] (0,0.5)--(0,-0.5);

     \end{tikzpicture}
      \caption{A extending of $x_{i}x_{i+1}$ and $y_{i}y_{i+1}$, where $m$ is even.}\label{extending}
   \end{figure}
    It is easy to see that the $(m,i)$-extender of $(CL_{n},\sigma)$ is isomorphic to a signed circular ladder with underlying graph $CL_{n+m}$.
    Additionally, the $(0,i)$-extender of $(CL_{n},\sigma)$ is simply $(CL_{n},\sigma)$.

    The following lemma shows that a $k$-NZF of $(CL_{n},\sigma)$ can be extended to a $k$-NZF of the $(4q,i)$-extender of $(CL_{n},\sigma)$ in certains cases, where $k\geq4$ and $q$ are integers.

    \begin{lem}\label{e-f}
      Let $k\geq 4$, $n\geq3$ and $q\geq 0$ be integers, and let $\sigma(x_{i}x_{i+1})=\sigma(y_{i}y_{i+1})=+1$ in $(CL_{n},\sigma)$, where the indices are considered modulo $n$.

      {\rm(1)} If there exists a $k$-NZF $f$ on $(CL_{n},\sigma)$ such that $f(x_{i}x_{i+1})=\pm1$ and $f(y_{i}y_{i+1})=\pm2$, then the $(4q,i)$-extender of $(CL_{n},\sigma)$ admits a $k$-NZF.

      {\rm(2)} If there exists a $k$-NZF $f$ on $(CL_{n},\sigma)$ such that $f(x_{i}x_{i+1})=\pm1$ and $f(y_{i}y_{i+1})=\pm1$, then the $(4q,i)$-extender of $(CL_{n},\sigma)$ admits a $k$-NZF.
    \end{lem}
    \begin{proof}
      We may assume that $q \geq 1$ since the statements hold trivially when $q = 0$.
      Denote the $(4q,i)$-extender of $(CL_{n},\sigma)$ by $(G,\tau)$.
      Let $P_{x_{i}}=x_{i}x^{i}_{1}x^{i}_{2}\cdots x^{i}_{4q}x_{i+1}$ and $P_{y_{i}}=y_{i}y^{i}_{1}y^{i}_{2}\cdots y^{i}_{4q}y_{i+1}$.
      Set $M^{*}=\{x^{i}_{j}y^{i}_{j}\in E(G):j\in[1,4q]\}$.
      There exists a $k$-flow of $(G,\tau)$, denoted by $f_{1}$, obtained from $f$, as follows.
      \begin{equation}
  f_{1}(e)=\begin{cases}
  f(e),&  e\in E(CL_{n})\setminus \{x_{i}x_{i+1},y_{i}y_{i+1}\};\\
  f(x_{i}x_{i+1}),&  e\in E(P_{x_{i}});\\
  f(y_{i}y_{i+1}),&  e\in E(P_{y_{i}});\\
  0,&  e\in M^{*}.\nonumber
  \end{cases}
  \end{equation}
   Namely, $supp(f_{1})=E(G)\setminus M^{*}$.
   Next, we construct another $3$-flow $f_{2}$ on $(G,\sigma)$ such that $M^{*} \subseteq supp(f_{2})$.
   The expression for $ f_{2} $ is detailed below, and we suggest readers refer to Fig. \ref{HL} for a visual representation to aid understanding.
     \begin{equation}
  f_{2}(e) =
\begin{cases}
1, & e \in M^{*}; \\
2, & e \in \{ x^{i}_{4l+2}x^{i}_{4l+3} : l \in [0, q-1] \}; \\
1, & e \in \{ x^{i}_{2l+1}x^{i}_{2l+2} : l \in [0, 2q-1] \}; \\
1, & e \in \{ y^{i}_{4l+1}y^{i}_{4l+2} : l \in [0, q-1] \}; \\
-1, & e \in \{ y^{i}_{4l+3}y^{i}_{4l+4} : l \in [0, q-1] \}; \\
0, & \text{otherwise.}\nonumber
\end{cases}
  \end{equation}

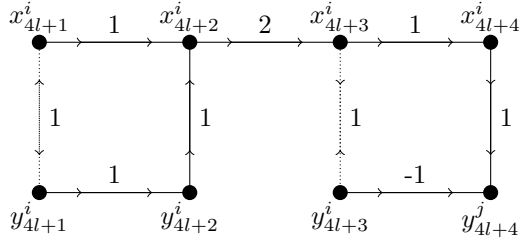
\begin{figure}[h]
  \centering
   \begin{tikzpicture}[scale=1]
  \draw[solid] (-3,2)--(3,2);
  \draw[densely dotted] (-3,2)--(-3,0);
  \draw[solid] (-3,0)--(-1,0);
  \draw[solid] (-1,0)--(-1,2);
  \draw[densely dotted] (1,2)--(1,0);
  \draw[solid] (1,0)--(3,0);
  \draw[solid] (3,0)--(3,2);

  \draw [>->] (-2.5,2)--(-1.5,2) node[above,midway]{1};
  \draw [densely dotted][<->] (-3,1.5)--(-3,0.5) node[right,midway]{1};
  \draw [>->] (-2.5,0)--(-1.5,0) node[above,midway]{1};
  \draw [>->] (-1,0.5)--(-1,1.5) node[right,midway]{1};
  \draw [>->] (-0.5,2)--(0.5,2) node[above,midway]{2};
  \draw [>->] (1.5,2)--(2.5,2) node[above,midway]{1};
  \draw [densely dotted][>-<] (1,1.5)--(1,0.5) node[right,midway]{1};
  \draw [>->] (1.5,0)--(2.5,0) node[above,midway]{-1};
  \draw [<-<] (3,0.5)--(3,1.5) node[right,midway]{1};

   \fill (-3,2) circle (.1);
   \node at (-3,2)[above]{$x^{i}_{4l+1}$};
   \fill (-3,0) circle (.1);
   \node at (-3,0) [below]{$y^{i}_{4l+1}$};
   \fill (-1,0) circle (.1);
   \node at (-1,0)[below]{$y^{i}_{4l+2}$};
   \fill (-1,2) circle (.1);
   \node at (-1,2)[above]{$x^{i}_{4l+2}$};

    \fill (1,2) circle (.1);
   \node at (1,2)[above]{$x^{i}_{4l+3}$};
   \fill (1,0) circle (.1);
   \node at (1,0)[below]{$y^{i}_{4l+3}$};
   \fill (3,0) circle (.1);
   \node at (3,0)[below]{$y^{j}_{4l+4}$};
   \fill (3,2) circle (.1);
   \node at (3,2)[above]{$x^{i}_{4l+4}$};

  \end{tikzpicture}
  \caption{A fragment of $f_{2}$.}\label{HL}
\end{figure}
  (1) Given $f(x_{i}x_{i+1})=\pm1$ and $f(y_{i}y_{i+1})=\pm2$, there are four cases that need to be considered.
  As illustrated in Fig. \ref{f1+f2}, for any $e\in E(P_{x})\cup E(P_{y})\cup M$, it holds that $\left|(f_{1}+f_{2})(e)\right|\leq 3$ or $\left|(f_{1}-f_{2})(e)\right|\leq 3$.
  Given that $f_{1}$ is a $k$-flow with $k\geq4$, it follows that $f_{1}+f_{2}$ or $f_{1}-f_{2}$ is a $k$-NZF on $(G,\tau)$.

  (2) In a similar manner, either $ f_{1} + 2f_{2} $ or $ f_{1} - 2f_{2} $ forms a $ k $-NZF on $(G, \tau)$, as depicted in Fig. \ref{f1+2f2}.

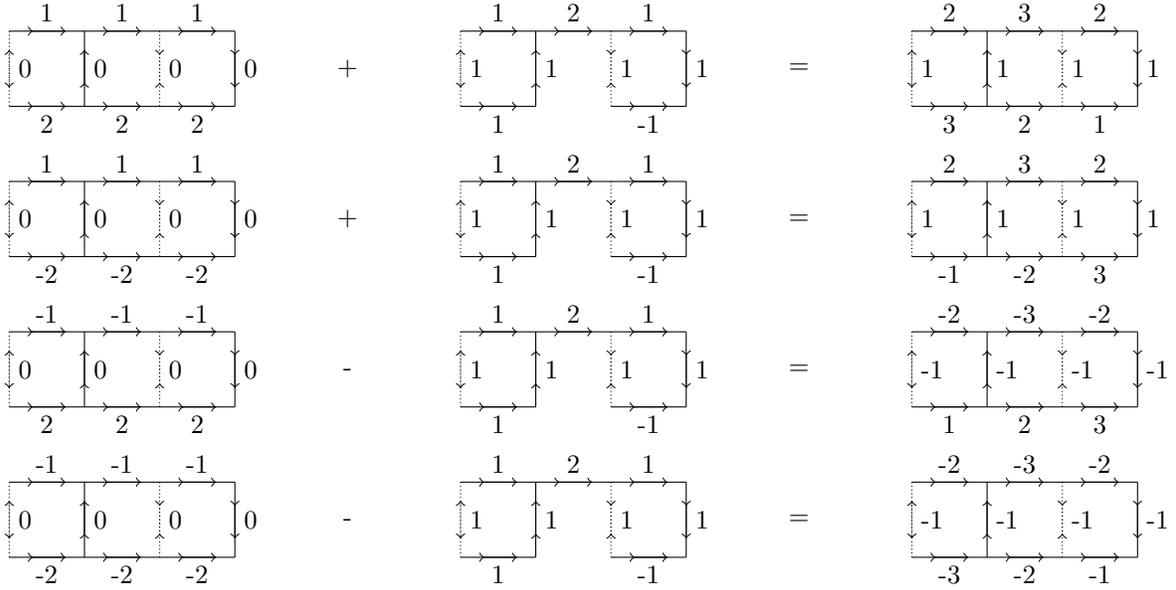
\begin{figure}[h]
  \centering
  \begin{tikzpicture}[scale=0.5]

  \draw[solid] (-15,12)--(-9,12);
  \draw[densely dotted] (-15,12)--(-15,10);
  \draw[solid] (-15,10)--(-9,10);
  \draw[solid] (-13,10)--(-13,12);
  \draw[densely dotted] (-11,12)--(-11,10);
  \draw[solid] (-9,10)--(-9,12);

  \draw [>->] (-14.5,12)--(-13.5,12) node[above,midway]{1};
  \draw [densely dotted][<->] (-15,11.5)--(-15,10.5) node[right,midway]{0};
  \draw [>->] (-14.5,10)--(-13.5,10) node[below,midway]{2};
  \draw [>->] (-13,10.5)--(-13,11.5) node[right,midway]{0};
  \draw [>->] (-12.5,12)--(-11.5,12) node[above,midway]{1};
  \draw [>->] (-12.5,10)--(-11.5,10) node[below,midway]{2};
  \draw [>->] (-10.5,12)--(-9.5,12) node[above,midway]{1};
  \draw [densely dotted][>-<] (-11,11.5)--(-11,10.5) node[right,midway]{0};
  \draw [>->] (-10.5,10)--(-9.5,10) node[below,midway]{2};
  \draw [<-<] (-9,10.5)--(-9,11.5) node[right,midway]{0};

  \node at (-6,11){+};

    \draw[solid] (-3,12)--(3,12);
  \draw[densely dotted] (-3,12)--(-3,10);
  \draw[solid] (-3,10)--(-1,10);
  \draw[solid] (-1,10)--(-1,12);
  \draw[densely dotted] (1,12)--(1,10);
  \draw[solid] (1,10)--(3,10);
  \draw[solid] (3,10)--(3,12);

  \draw [>->] (-2.5,12)--(-1.5,12) node[above,midway]{1};
  \draw [densely dotted][<->] (-3,11.5)--(-3,10.5) node[right,midway]{1};
  \draw [>->] (-2.5,10)--(-1.5,10) node[below,midway]{1};
  \draw [>->] (-1,10.5)--(-1,11.5) node[right,midway]{1};
  \draw [>->] (-0.5,12)--(0.5,12) node[above,midway]{2};
  \draw [>->] (1.5,12)--(2.5,12) node[above,midway]{1};
  \draw [densely dotted][>-<] (1,11.5)--(1,10.5) node[right,midway]{1};
  \draw [>->] (1.5,10)--(2.5,10) node[below,midway]{-1};
  \draw [<-<] (3,10.5)--(3,11.5) node[right,midway]{1};

    \node at (6,11){=};

    \draw[solid] (9,12)--(15,12);
  \draw[densely dotted] (9,12)--(9,10);
  \draw[solid] (9,10)--(15,10);
  \draw[solid] (11,10)--(11,12);
  \draw[densely dotted] (13,12)--(13,10);
  \draw[solid] (15,10)--(15,12);

  \draw [>->] (9.5,12)--(10.5,12) node[above,midway]{2};
  \draw [densely dotted][<->] (9,11.5)--(9,10.5) node[right,midway]{1};
  \draw [>->] (9.5,10)--(10.5,10) node[below,midway]{3};
  \draw [>->] (11,10.5)--(11,11.5) node[right,midway]{1};
  \draw [>->] (11.5,12)--(12.5,12) node[above,midway]{3};
  \draw [>->] (11.5,10)--(12.5,10) node[below,midway]{2};
  \draw [>->] (13.5,12)--(14.5,12) node[above,midway]{2};
  \draw [densely dotted][>-<] (13,11.5)--(13,10.5) node[right,midway]{1};
  \draw [>->] (13.5,10)--(14.5,10) node[below,midway]{1};
  \draw [<-<] (15,10.5)--(15,11.5) node[right,midway]{1};

  \draw[solid] (-15,8)--(-9,8);
  \draw[densely dotted] (-15,8)--(-15,6);
  \draw[solid] (-15,6)--(-9,6);
  \draw[solid] (-13,6)--(-13,8);
  \draw[densely dotted] (-11,8)--(-11,6);
  \draw[solid] (-9,6)--(-9,8);

  \draw [>->] (-14.5,8)--(-13.5,8) node[above,midway]{1};
  \draw [densely dotted][<->] (-15,7.5)--(-15,6.5) node[right,midway]{0};
  \draw [>->] (-14.5,6)--(-13.5,6) node[below,midway]{-2};
  \draw [>->] (-13,6.5)--(-13,7.5) node[right,midway]{0};
  \draw [>->] (-12.5,8)--(-11.5,8) node[above,midway]{1};
  \draw [>->] (-12.5,6)--(-11.5,6) node[below,midway]{-2};
  \draw [>->] (-10.5,8)--(-9.5,8) node[above,midway]{1};
  \draw [densely dotted][>-<] (-11,7.5)--(-11,6.5) node[right,midway]{0};
  \draw [>->] (-10.5,6)--(-9.5,6) node[below,midway]{-2};
  \draw [<-<] (-9,6.5)--(-9,7.5) node[right,midway]{0};

  \node at (-6,7){+};

    \draw[solid] (-3,8)--(3,8);
  \draw[densely dotted] (-3,8)--(-3,6);
  \draw[solid] (-3,6)--(-1,6);
  \draw[solid] (-1,6)--(-1,8);
  \draw[densely dotted] (1,8)--(1,6);
  \draw[solid] (1,6)--(3,6);
  \draw[solid] (3,6)--(3,8);

  \draw [>->] (-2.5,8)--(-1.5,8) node[above,midway]{1};
  \draw [densely dotted][<->] (-3,7.5)--(-3,6.5) node[right,midway]{1};
  \draw [>->] (-2.5,6)--(-1.5,6) node[below,midway]{1};
  \draw [>->] (-1,6.5)--(-1,7.5) node[right,midway]{1};
  \draw [>->] (-0.5,8)--(0.5,8) node[above,midway]{2};
  \draw [>->] (1.5,8)--(2.5,8) node[above,midway]{1};
  \draw [densely dotted][>-<] (1,7.5)--(1,6.5) node[right,midway]{1};
  \draw [>->] (1.5,6)--(2.5,6) node[below,midway]{-1};
  \draw [<-<] (3,6.5)--(3,7.5) node[right,midway]{1};

    \node at (6,7){=};

    \draw[solid] (9,8)--(15,8);
  \draw[densely dotted] (9,8)--(9,6);
  \draw[solid] (9,6)--(15,6);
  \draw[solid] (11,6)--(11,8);
  \draw[densely dotted] (13,8)--(13,6);
  \draw[solid] (15,6)--(15,8);

  \draw [>->] (9.5,8)--(10.5,8) node[above,midway]{2};
  \draw [densely dotted][<->] (9,7.5)--(9,6.5) node[right,midway]{1};
  \draw [>->] (9.5,6)--(10.5,6) node[below,midway]{-1};
  \draw [>->] (11,6.5)--(11,7.5) node[right,midway]{1};
  \draw [>->] (11.5,8)--(12.5,8) node[above,midway]{3};
  \draw [>->] (11.5,6)--(12.5,6) node[below,midway]{-2};
  \draw [>->] (13.5,8)--(14.5,8) node[above,midway]{2};
  \draw [densely dotted][>-<] (13,7.5)--(13,6.5) node[right,midway]{1};
  \draw [>->] (13.5,6)--(14.5,6) node[below,midway]{3};
  \draw [<-<] (15,6.5)--(15,7.5) node[right,midway]{1};


  \draw[solid] (-15,4)--(-9,4);
  \draw[densely dotted] (-15,4)--(-15,2);
  \draw[solid] (-15,2)--(-9,2);
  \draw[solid] (-13,2)--(-13,4);
  \draw[densely dotted] (-11,4)--(-11,2);
  \draw[solid] (-9,2)--(-9,4);

  \draw [>->] (-14.5,4)--(-13.5,4) node[above,midway]{-1};
  \draw [densely dotted][<->] (-15,3.5)--(-15,2.5) node[right,midway]{0};
  \draw [>->] (-14.5,2)--(-13.5,2) node[below,midway]{2};
  \draw [>->] (-13,2.5)--(-13,3.5) node[right,midway]{0};
  \draw [>->] (-12.5,4)--(-11.5,4) node[above,midway]{-1};
  \draw [>->] (-12.5,2)--(-11.5,2) node[below,midway]{2};
  \draw [>->] (-10.5,4)--(-9.5,4) node[above,midway]{-1};
  \draw [densely dotted][>-<] (-11,3.5)--(-11,2.5) node[right,midway]{0};
  \draw [>->] (-10.5,2)--(-9.5,2) node[below,midway]{2};
  \draw [<-<] (-9,2.5)--(-9,3.5) node[right,midway]{0};

  \node at (-6,3){-};

    \draw[solid] (-3,4)--(3,4);
  \draw[densely dotted] (-3,4)--(-3,2);
  \draw[solid] (-3,2)--(-1,2);
  \draw[solid] (-1,2)--(-1,4);
  \draw[densely dotted] (1,4)--(1,2);
  \draw[solid] (1,2)--(3,2);
  \draw[solid] (3,2)--(3,4);

  \draw [>->] (-2.5,4)--(-1.5,4) node[above,midway]{1};
  \draw [densely dotted][<->] (-3,3.5)--(-3,2.5) node[right,midway]{1};
  \draw [>->] (-2.5,2)--(-1.5,2) node[below,midway]{1};
  \draw [>->] (-1,2.5)--(-1,3.5) node[right,midway]{1};
  \draw [>->] (-0.5,4)--(0.5,4) node[above,midway]{2};
  \draw [>->] (1.5,4)--(2.5,4) node[above,midway]{1};
  \draw [densely dotted][>-<] (1,3.5)--(1,2.5) node[right,midway]{1};
  \draw [>->] (1.5,2)--(2.5,2) node[below,midway]{-1};
  \draw [<-<] (3,2.5)--(3,3.5) node[right,midway]{1};

    \node at (6,3){=};

    \draw[solid] (9,4)--(15,4);
  \draw[densely dotted] (9,4)--(9,2);
  \draw[solid] (9,2)--(15,2);
  \draw[solid] (11,2)--(11,4);
  \draw[densely dotted] (13,4)--(13,2);
  \draw[solid] (15,2)--(15,4);

  \draw [>->] (9.5,4)--(10.5,4) node[above,midway]{-2};
  \draw [densely dotted][<->] (9,3.5)--(9,2.5) node[right,midway]{-1};
  \draw [>->] (9.5,2)--(10.5,2) node[below,midway]{1};
  \draw [>->] (11,2.5)--(11,3.5) node[right,midway]{-1};
  \draw [>->] (11.5,4)--(12.5,4) node[above,midway]{-3};
  \draw [>->] (11.5,2)--(12.5,2) node[below,midway]{2};
  \draw [>->] (13.5,4)--(14.5,4) node[above,midway]{-2};
  \draw [densely dotted][>-<] (13,3.5)--(13,2.5) node[right,midway]{-1};
  \draw [>->] (13.5,2)--(14.5,2) node[below,midway]{3};
  \draw [<-<] (15,2.5)--(15,3.5) node[right,midway]{-1};


  \draw[solid] (-15,0)--(-9,0);
  \draw[densely dotted] (-15,0)--(-15,-2);
  \draw[solid] (-15,-2)--(-9,-2);
  \draw[solid] (-13,-2)--(-13,0);
  \draw[densely dotted] (-11,0)--(-11,-2);
  \draw[solid] (-9,-2)--(-9,0);

  \draw [>->] (-14.5,0)--(-13.5,0) node[above,midway]{-1};
  \draw [densely dotted][<->] (-15,-0.5)--(-15,-1.5) node[right,midway]{0};
  \draw [>->] (-14.5,-2)--(-13.5,-2) node[below,midway]{-2};
  \draw [>->] (-13,-1.5)--(-13,-0.5) node[right,midway]{0};
  \draw [>->] (-12.5,0)--(-11.5,0) node[above,midway]{-1};
  \draw [>->] (-12.5,-2)--(-11.5,-2) node[below,midway]{-2};
  \draw [>->] (-10.5,0)--(-9.5,0) node[above,midway]{-1};
  \draw [densely dotted][>-<] (-11,-0.5)--(-11,-1.5) node[right,midway]{0};
  \draw [>->] (-10.5,-2)--(-9.5,-2) node[below,midway]{-2};
  \draw [<-<] (-9,-1.5)--(-9,-0.5) node[right,midway]{0};

  \node at (-6,-1){-};

    \draw[solid] (-3,0)--(3,0);
  \draw[densely dotted] (-3,0)--(-3,-2);
  \draw[solid] (-3,-2)--(-1,-2);
  \draw[solid] (-1,-2)--(-1,0);
  \draw[densely dotted] (1,0)--(1,-2);
  \draw[solid] (1,-2)--(3,-2);
  \draw[solid] (3,-2)--(3,0);

  \draw [>->] (-2.5,0)--(-1.5,0) node[above,midway]{1};
  \draw [densely dotted][<->] (-3,-0.5)--(-3,-1.5) node[right,midway]{1};
  \draw [>->] (-2.5,-2)--(-1.5,-2) node[below,midway]{1};
  \draw [>->] (-1,-1.5)--(-1,-0.5) node[right,midway]{1};
  \draw [>->] (-0.5,0)--(0.5,0) node[above,midway]{2};
  \draw [>->] (1.5,0)--(2.5,0) node[above,midway]{1};
  \draw [densely dotted][>-<] (1,-0.5)--(1,-1.5) node[right,midway]{1};
  \draw [>->] (1.5,-2)--(2.5,-2) node[below,midway]{-1};
  \draw [<-<] (3,-1.5)--(3,-0.5) node[right,midway]{1};

    \node at (6,-1){=};

    \draw[solid] (9,0)--(15,0);
  \draw[densely dotted] (9,0)--(9,-2);
  \draw[solid] (9,-2)--(15,-2);
  \draw[solid] (11,-2)--(11,0);
  \draw[densely dotted] (13,0)--(13,-2);
  \draw[solid] (15,-2)--(15,0);

  \draw [>->] (9.5,0)--(10.5,0) node[above,midway]{-2};
  \draw [densely dotted][<->] (9,-0.5)--(9,-1.5) node[right,midway]{-1};
  \draw [>->] (9.5,-2)--(10.5,-2) node[below,midway]{-3};
  \draw [>->] (11,-1.5)--(11,-0.5) node[right,midway]{-1};
  \draw [>->] (11.5,0)--(12.5,0) node[above,midway]{-3};
  \draw [>->] (11.5,-2)--(12.5,-2) node[below,midway]{-2};
  \draw [>->] (13.5,0)--(14.5,0) node[above,midway]{-2};
  \draw [densely dotted][>-<] (13,-0.5)--(13,-1.5) node[right,midway]{-1};
  \draw [>->] (13.5,-2)--(14.5,-2) node[below,midway]{-1};
  \draw [<-<] (15,-1.5)--(15,-0.5) node[right,midway]{-1};
  \end{tikzpicture}
  \caption{A fragment of $f_{1}+f_{2}$ and $f_{1}-f_{2}$.}\label{f1+f2}
\end{figure}
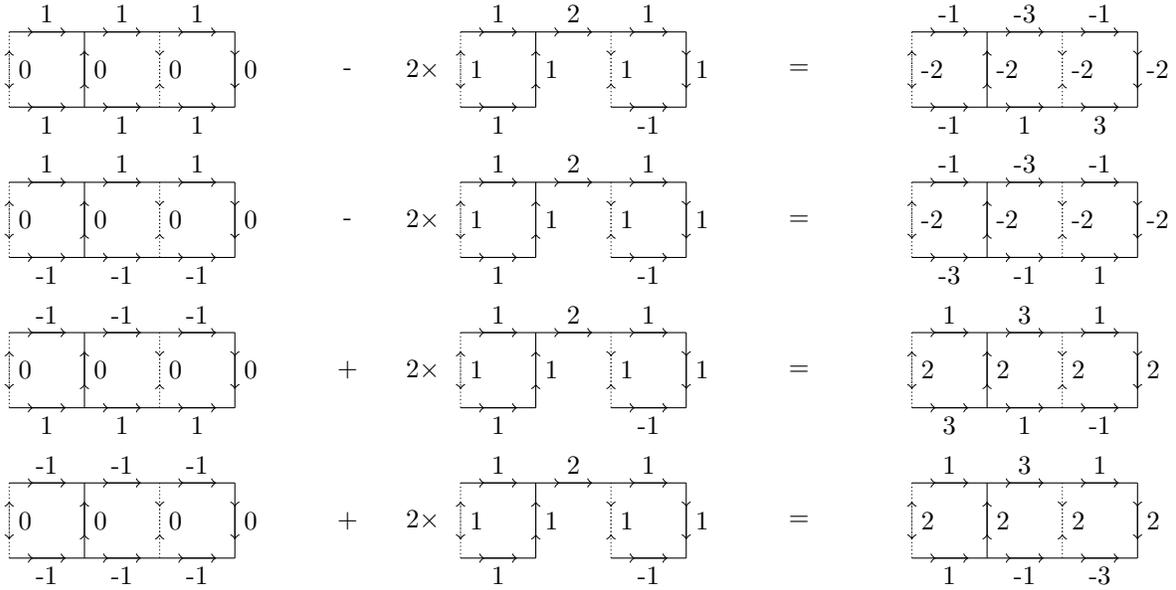
\begin{figure}[h]
  \centering
  \begin{tikzpicture}[scale=0.5]

  \draw[solid] (-15,12)--(-9,12);
  \draw[densely dotted] (-15,12)--(-15,10);
  \draw[solid] (-15,10)--(-9,10);
  \draw[solid] (-13,10)--(-13,12);
  \draw[densely dotted] (-11,12)--(-11,10);
  \draw[solid] (-9,10)--(-9,12);

  \draw [>->] (-14.5,12)--(-13.5,12) node[above,midway]{1};
  \draw [densely dotted][<->] (-15,11.5)--(-15,10.5) node[right,midway]{0};
  \draw [>->] (-14.5,10)--(-13.5,10) node[below,midway]{1};
  \draw [>->] (-13,10.5)--(-13,11.5) node[right,midway]{0};
  \draw [>->] (-12.5,12)--(-11.5,12) node[above,midway]{1};
  \draw [>->] (-12.5,10)--(-11.5,10) node[below,midway]{1};
  \draw [>->] (-10.5,12)--(-9.5,12) node[above,midway]{1};
  \draw [densely dotted][>-<] (-11,11.5)--(-11,10.5) node[right,midway]{0};
  \draw [>->] (-10.5,10)--(-9.5,10) node[below,midway]{1};
  \draw [<-<] (-9,10.5)--(-9,11.5) node[right,midway]{0};

  \node at (-6,11){-};
  \node at (-4,11){2$\times$};

    \draw[solid] (-3,12)--(3,12);
  \draw[densely dotted] (-3,12)--(-3,10);
  \draw[solid] (-3,10)--(-1,10);
  \draw[solid] (-1,10)--(-1,12);
  \draw[densely dotted] (1,12)--(1,10);
  \draw[solid] (1,10)--(3,10);
  \draw[solid] (3,10)--(3,12);

  \draw [>->] (-2.5,12)--(-1.5,12) node[above,midway]{1};
  \draw [densely dotted][<->] (-3,11.5)--(-3,10.5) node[right,midway]{1};
  \draw [>->] (-2.5,10)--(-1.5,10) node[below,midway]{1};
  \draw [>->] (-1,10.5)--(-1,11.5) node[right,midway]{1};
  \draw [>->] (-0.5,12)--(0.5,12) node[above,midway]{2};
  \draw [>->] (1.5,12)--(2.5,12) node[above,midway]{1};
  \draw [densely dotted][>-<] (1,11.5)--(1,10.5) node[right,midway]{1};
  \draw [>->] (1.5,10)--(2.5,10) node[below,midway]{-1};
  \draw [<-<] (3,10.5)--(3,11.5) node[right,midway]{1};

    \node at (6,11){=};

    \draw[solid] (9,12)--(15,12);
  \draw[densely dotted] (9,12)--(9,10);
  \draw[solid] (9,10)--(15,10);
  \draw[solid] (11,10)--(11,12);
  \draw[densely dotted] (13,12)--(13,10);
  \draw[solid] (15,10)--(15,12);

  \draw [>->] (9.5,12)--(10.5,12) node[above,midway]{-1};
  \draw [densely dotted][<->] (9,11.5)--(9,10.5) node[right,midway]{-2};
  \draw [>->] (9.5,10)--(10.5,10) node[below,midway]{-1};
  \draw [>->] (11,10.5)--(11,11.5) node[right,midway]{-2};
  \draw [>->] (11.5,12)--(12.5,12) node[above,midway]{-3};
  \draw [>->] (11.5,10)--(12.5,10) node[below,midway]{1};
  \draw [>->] (13.5,12)--(14.5,12) node[above,midway]{-1};
  \draw [densely dotted][>-<] (13,11.5)--(13,10.5) node[right,midway]{-2};
  \draw [>->] (13.5,10)--(14.5,10) node[below,midway]{3};
  \draw [<-<] (15,10.5)--(15,11.5) node[right,midway]{-2};

  \draw[solid] (-15,8)--(-9,8);
  \draw[densely dotted] (-15,8)--(-15,6);
  \draw[solid] (-15,6)--(-9,6);
  \draw[solid] (-13,6)--(-13,8);
  \draw[densely dotted] (-11,8)--(-11,6);
  \draw[solid] (-9,6)--(-9,8);

  \draw [>->] (-14.5,8)--(-13.5,8) node[above,midway]{1};
  \draw [densely dotted][<->] (-15,7.5)--(-15,6.5) node[right,midway]{0};
  \draw [>->] (-14.5,6)--(-13.5,6) node[below,midway]{-1};
  \draw [>->] (-13,6.5)--(-13,7.5) node[right,midway]{0};
  \draw [>->] (-12.5,8)--(-11.5,8) node[above,midway]{1};
  \draw [>->] (-12.5,6)--(-11.5,6) node[below,midway]{-1};
  \draw [>->] (-10.5,8)--(-9.5,8) node[above,midway]{1};
  \draw [densely dotted][>-<] (-11,7.5)--(-11,6.5) node[right,midway]{0};
  \draw [>->] (-10.5,6)--(-9.5,6) node[below,midway]{-1};
  \draw [<-<] (-9,6.5)--(-9,7.5) node[right,midway]{0};

  \node at (-6,7){-};
  \node at (-4,7){2$\times$};

    \draw[solid] (-3,8)--(3,8);
  \draw[densely dotted] (-3,8)--(-3,6);
  \draw[solid] (-3,6)--(-1,6);
  \draw[solid] (-1,6)--(-1,8);
  \draw[densely dotted] (1,8)--(1,6);
  \draw[solid] (1,6)--(3,6);
  \draw[solid] (3,6)--(3,8);

  \draw [>->] (-2.5,8)--(-1.5,8) node[above,midway]{1};
  \draw [densely dotted][<->] (-3,7.5)--(-3,6.5) node[right,midway]{1};
  \draw [>->] (-2.5,6)--(-1.5,6) node[below,midway]{1};
  \draw [>->] (-1,6.5)--(-1,7.5) node[right,midway]{1};
  \draw [>->] (-0.5,8)--(0.5,8) node[above,midway]{2};
  \draw [>->] (1.5,8)--(2.5,8) node[above,midway]{1};
  \draw [densely dotted][>-<] (1,7.5)--(1,6.5) node[right,midway]{1};
  \draw [>->] (1.5,6)--(2.5,6) node[below,midway]{-1};
  \draw [<-<] (3,6.5)--(3,7.5) node[right,midway]{1};

    \node at (6,7){=};

    \draw[solid] (9,8)--(15,8);
  \draw[densely dotted] (9,8)--(9,6);
  \draw[solid] (9,6)--(15,6);
  \draw[solid] (11,6)--(11,8);
  \draw[densely dotted] (13,8)--(13,6);
  \draw[solid] (15,6)--(15,8);

  \draw [>->] (9.5,8)--(10.5,8) node[above,midway]{-1};
  \draw [densely dotted][<->] (9,7.5)--(9,6.5) node[right,midway]{-2};
  \draw [>->] (9.5,6)--(10.5,6) node[below,midway]{-3};
  \draw [>->] (11,6.5)--(11,7.5) node[right,midway]{-2};
  \draw [>->] (11.5,8)--(12.5,8) node[above,midway]{-3};
  \draw [>->] (11.5,6)--(12.5,6) node[below,midway]{-1};
  \draw [>->] (13.5,8)--(14.5,8) node[above,midway]{-1};
  \draw [densely dotted][>-<] (13,7.5)--(13,6.5) node[right,midway]{-2};
  \draw [>->] (13.5,6)--(14.5,6) node[below,midway]{1};
  \draw [<-<] (15,6.5)--(15,7.5) node[right,midway]{-2};


  \draw[solid] (-15,4)--(-9,4);
  \draw[densely dotted] (-15,4)--(-15,2);
  \draw[solid] (-15,2)--(-9,2);
  \draw[solid] (-13,2)--(-13,4);
  \draw[densely dotted] (-11,4)--(-11,2);
  \draw[solid] (-9,2)--(-9,4);

  \draw [>->] (-14.5,4)--(-13.5,4) node[above,midway]{-1};
  \draw [densely dotted][<->] (-15,3.5)--(-15,2.5) node[right,midway]{0};
  \draw [>->] (-14.5,2)--(-13.5,2) node[below,midway]{1};
  \draw [>->] (-13,2.5)--(-13,3.5) node[right,midway]{0};
  \draw [>->] (-12.5,4)--(-11.5,4) node[above,midway]{-1};
  \draw [>->] (-12.5,2)--(-11.5,2) node[below,midway]{1};
  \draw [>->] (-10.5,4)--(-9.5,4) node[above,midway]{-1};
  \draw [densely dotted][>-<] (-11,3.5)--(-11,2.5) node[right,midway]{0};
  \draw [>->] (-10.5,2)--(-9.5,2) node[below,midway]{1};
  \draw [<-<] (-9,2.5)--(-9,3.5) node[right,midway]{0};

  \node at (-6,3){+};
  \node at (-4,3){2$\times$};
    \draw[solid] (-3,4)--(3,4);
  \draw[densely dotted] (-3,4)--(-3,2);
  \draw[solid] (-3,2)--(-1,2);
  \draw[solid] (-1,2)--(-1,4);
  \draw[densely dotted] (1,4)--(1,2);
  \draw[solid] (1,2)--(3,2);
  \draw[solid] (3,2)--(3,4);

  \draw [>->] (-2.5,4)--(-1.5,4) node[above,midway]{1};
  \draw [densely dotted][<->] (-3,3.5)--(-3,2.5) node[right,midway]{1};
  \draw [>->] (-2.5,2)--(-1.5,2) node[below,midway]{1};
  \draw [>->] (-1,2.5)--(-1,3.5) node[right,midway]{1};
  \draw [>->] (-0.5,4)--(0.5,4) node[above,midway]{2};
  \draw [>->] (1.5,4)--(2.5,4) node[above,midway]{1};
  \draw [densely dotted][>-<] (1,3.5)--(1,2.5) node[right,midway]{1};
  \draw [>->] (1.5,2)--(2.5,2) node[below,midway]{-1};
  \draw [<-<] (3,2.5)--(3,3.5) node[right,midway]{1};

    \node at (6,3){=};

    \draw[solid] (9,4)--(15,4);
  \draw[densely dotted] (9,4)--(9,2);
  \draw[solid] (9,2)--(15,2);
  \draw[solid] (11,2)--(11,4);
  \draw[densely dotted] (13,4)--(13,2);
  \draw[solid] (15,2)--(15,4);

  \draw [>->] (9.5,4)--(10.5,4) node[above,midway]{1};
  \draw [densely dotted][<->] (9,3.5)--(9,2.5) node[right,midway]{2};
  \draw [>->] (9.5,2)--(10.5,2) node[below,midway]{3};
  \draw [>->] (11,2.5)--(11,3.5) node[right,midway]{2};
  \draw [>->] (11.5,4)--(12.5,4) node[above,midway]{3};
  \draw [>->] (11.5,2)--(12.5,2) node[below,midway]{1};
  \draw [>->] (13.5,4)--(14.5,4) node[above,midway]{1};
  \draw [densely dotted][>-<] (13,3.5)--(13,2.5) node[right,midway]{2};
  \draw [>->] (13.5,2)--(14.5,2) node[below,midway]{-1};
  \draw [<-<] (15,2.5)--(15,3.5) node[right,midway]{2};


  \draw[solid] (-15,0)--(-9,0);
  \draw[densely dotted] (-15,0)--(-15,-2);
  \draw[solid] (-15,-2)--(-9,-2);
  \draw[solid] (-13,-2)--(-13,0);
  \draw[densely dotted] (-11,0)--(-11,-2);
  \draw[solid] (-9,-2)--(-9,0);

  \draw [>->] (-14.5,0)--(-13.5,0) node[above,midway]{-1};
  \draw [densely dotted][<->] (-15,-0.5)--(-15,-1.5) node[right,midway]{0};
  \draw [>->] (-14.5,-2)--(-13.5,-2) node[below,midway]{-1};
  \draw [>->] (-13,-1.5)--(-13,-0.5) node[right,midway]{0};
  \draw [>->] (-12.5,0)--(-11.5,0) node[above,midway]{-1};
  \draw [>->] (-12.5,-2)--(-11.5,-2) node[below,midway]{-1};
  \draw [>->] (-10.5,0)--(-9.5,0) node[above,midway]{-1};
  \draw [densely dotted][>-<] (-11,-0.5)--(-11,-1.5) node[right,midway]{0};
  \draw [>->] (-10.5,-2)--(-9.5,-2) node[below,midway]{-1};
  \draw [<-<] (-9,-1.5)--(-9,-0.5) node[right,midway]{0};

  \node at (-6,-1){+};
  \node at (-4,-1){2$\times$};

    \draw[solid] (-3,0)--(3,0);
  \draw[densely dotted] (-3,0)--(-3,-2);
  \draw[solid] (-3,-2)--(-1,-2);
  \draw[solid] (-1,-2)--(-1,0);
  \draw[densely dotted] (1,0)--(1,-2);
  \draw[solid] (1,-2)--(3,-2);
  \draw[solid] (3,-2)--(3,0);

  \draw [>->] (-2.5,0)--(-1.5,0) node[above,midway]{1};
  \draw [densely dotted][<->] (-3,-0.5)--(-3,-1.5) node[right,midway]{1};
  \draw [>->] (-2.5,-2)--(-1.5,-2) node[below,midway]{1};
  \draw [>->] (-1,-1.5)--(-1,-0.5) node[right,midway]{1};
  \draw [>->] (-0.5,0)--(0.5,0) node[above,midway]{2};
  \draw [>->] (1.5,0)--(2.5,0) node[above,midway]{1};
  \draw [densely dotted][>-<] (1,-0.5)--(1,-1.5) node[right,midway]{1};
  \draw [>->] (1.5,-2)--(2.5,-2) node[below,midway]{-1};
  \draw [<-<] (3,-1.5)--(3,-0.5) node[right,midway]{1};

    \node at (6,-1){=};

    \draw[solid] (9,0)--(15,0);
  \draw[densely dotted] (9,0)--(9,-2);
  \draw[solid] (9,-2)--(15,-2);
  \draw[solid] (11,-2)--(11,0);
  \draw[densely dotted] (13,0)--(13,-2);
  \draw[solid] (15,-2)--(15,0);

  \draw [>->] (9.5,0)--(10.5,0) node[above,midway]{1};
  \draw [densely dotted][<->] (9,-0.5)--(9,-1.5) node[right,midway]{2};
  \draw [>->] (9.5,-2)--(10.5,-2) node[below,midway]{1};
  \draw [>->] (11,-1.5)--(11,-0.5) node[right,midway]{2};
  \draw [>->] (11.5,0)--(12.5,0) node[above,midway]{3};
  \draw [>->] (11.5,-2)--(12.5,-2) node[below,midway]{-1};
  \draw [>->] (13.5,0)--(14.5,0) node[above,midway]{1};
  \draw [densely dotted][>-<] (13,-0.5)--(13,-1.5) node[right,midway]{2};
  \draw [>->] (13.5,-2)--(14.5,-2) node[below,midway]{-3};
  \draw [<-<] (15,-1.5)--(15,-0.5) node[right,midway]{2};
  \end{tikzpicture}
  \caption{A fragment of $f_{1}+2f_{2}$ and $f_{1}-2f_{2}$.}\label{f1+2f2}
\end{figure}
    \end{proof}

    The following theorem shows that every flow-admissible signed circular ladder admits a $6$-NZF.

    \begin{thm}\label{C-6}
      Every flow-admissible $(CL_{n},\sigma)$ admits a $6$-NZF.
   \end{thm}
   \begin{proof}

     We consider three cases based on the sign of $\sigma(C_{x})$ and $\sigma(C_{y})$.\\

     \textbf{Case 1.}
     $\sigma(C_{x})=\sigma(C_{y})=+1$.

     Without loss of generality, we assume that $(C_{x},\sigma)$ and $(C_{y},\sigma)$ are all-positive; otherwise we switch at some vertex in $V(C_{x})\cup V(C_{y})$ to ensure that every edge in $C_{x}$ and $C_{y}$ is positive.
     Since $(CL_{n},\sigma)$ is flow-admissible, it follows that $n\geq2$.\\

      \textbf{Subcase 1.1.} There exists an $i\in [0,n-1]$ such that $\sigma(x_{i}y_{i})=\sigma(x_{i+1}y_{i+1})$ modulo $n$.

      Assume that $\sigma(x_{i}y_{i})=\sigma(x_{i+1}y_{i+1})=+1$; otherwise, we switch at $V(C_{x})$.
      Then, $(C_{x}\cup C_{y}\cup\{x_{i}y_{i}, x_{i+1}y_{i+1}\})\setminus \{x_{i}x_{i+1}, y_{i}y_{i+1}\}$ forms an all-positive Hamiltonian circuit.
      By Theorem \ref{BH-6-flow}, we conclude that $(CL_{n},\sigma)$ admits a $6$-NZF.\\

      \textbf{Subcase 1.2.} There is no $i\in [0,n-1]$ such that $\sigma(x_{i}y_{i})=\sigma(x_{i+1}y_{i+1})$ modulo $n$.

      For any $i\in [0,n-1]$, we have $\sigma(x_{i}y_{i})\neq\sigma(x_{i+1}y_{i+1})$.
      Thus, $n$ is even; otherwise, there exists a $j\in [0,n-1]$ such that $\sigma(x_{j}y_{j})=\sigma(x_{j+1}y_{j+1})$ modulo $n$.
      Because $(CL_{n},\sigma)$ is flow-admissible, it follows that  $n\geq 4$.
      Otherwise, there would be only one negative edge, a contradiction.
      Assume that $\sigma(x_{0}y_{0})=+1$.
      Otherwise, perform a switching at $V(C_{1})$.

      \begin{figure}[h]

  \begin{minipage}[b]{.5\textwidth}
  \centering
  \begin{tikzpicture}[scale=0.8]


    \draw[solid] (-2,2)--(2,2) ;
    \draw [<-] (-1,2)--(0,2)node[above]{2} ;
    \draw [<-] (1,2)--(2,2) ;

    \draw[solid] (-2,2)--(-2,-2) ;
    \draw [>-] (-2,1)--(-2,0) node[left]{1} ;
    \draw [>-] (-2,-1)--(-2,-2) ;

    \draw[solid] (-2,-2)--(2,-2);
    \draw [>-] (-1,-2)--(0,-2)node[below]{2};
    \draw [>-] (1,-2)--(2,-2);

    \draw[solid] (2,2)--(2,-2);
    \draw [<-] (2,1)--(2,0)node[right]{3} ;
    \draw [<-] (2,-1)--(2,-2);

    \draw[solid] (-1,1)--(1,1);
    \draw [<-] (-0.5,1)--(0,1) node[above]{1} ;
    \draw [<-] (0.5,1)--(1,1);

    \draw[solid] (-1,1)--(-1,-1)  ;
    \draw [>-] (-1,0.5)--(-1,0)node[left]{2} ;
    \draw [>-] (-1,-0.5)--(-1,-1);

    \draw[solid] (-1,-1)--(1,-1) ;
    \draw [>-] (-0.5,-1)--(0,-1) node[below]{3};
    \draw [>-] (0.5,-1)--(1,-1);

    \draw[solid] (1,1)--(1,-1);
    \draw [<-] (1,0.5)--(1,0)node[right]{2};
    \draw [<-] (1,-0.5)--(1,-1);

    \draw [->] (-2,2)--(-1.75,1.75) ;
    \draw [solid](-2,2)--(-1,1) node[left,midway]{1};
    \draw [->] (-1.5,1.5)--(-1.25,1.25) ;

     \draw [densely dotted] (-2,-2)--(-1,-1);
    \draw [densely dotted][<->] (-1.75,-1.75)--(-1.25,-1.25)node[left]{1} ;

    \draw [-<] (2,-2)--(1.75,-1.75);
    \draw [solid](2,-2)--(1,-1) node[left,midway]{1} ;
    \draw [-<] (1.5,-1.5)--(1.25,-1.25);

    \draw [densely dotted]  (2,2)--(1,1);
    \draw [densely dotted] [>-<] (1.75,1.75)--(1.25,1.25)node[right]{1};

   \fill (-2,2) circle (.1);
   \node at (-2,2)[above]{$y_{0}$};
   \fill (-2,-2) circle (.1);
   \fill (2,-2) circle (.1);
   \fill (2,2) circle (.1);
   \node at (2,2)[above]{$y_{1}$};

    \fill (-1,1) circle (.1);
   \node at (-0.75,0.75){$x_{0}$};
   \fill (-1,-1) circle (.1);
   \fill (1,-1) circle (.1);
   \fill (1,1) circle (.1);
   \node at (0.75,0.75){$x_{1}$};

  \end{tikzpicture}
  \caption{A $4$-NZF $f_{1}$ on $(CL_{4},\sigma_{1})$.}\label{CL4}
  \end{minipage}%
  \begin{minipage}[b]{.5\textwidth}
  \centering
  \begin{tikzpicture}[scale=0.4]

    \draw[solid] (5,0)--(3,4) ;
    \draw [<-<] (3.5,3)--(4.5,1)node[right,above,midway]{3};

    \draw[solid] (3,4)--(-3,4) ;
    \draw [<-<] (-1.5,4)--(1.5,4) node[above,midway]{1} ;

    \draw[solid] (-3,4)--(-5,0);
    \draw [<-<] (-3.5,3)--(-4.5,1)node[left,midway]{2};

    \draw[solid] (-5,0)--(-3,-4);
    \draw [>->] (-3.5,-3)--(-4.5,-1)node[left,midway]{1} ;

    \draw[solid] (-3,-4)--(3,-4);
    \draw [>->] (-1.5,-4)--(1.5,-4)node[midway,below]{1} ;

    \draw[solid] (3,-4)--(5,0);
    \draw [>->] (3.5,-3)--(4.5,-1)node[midway,right]{2} ;

    \draw[solid] (3,0)--(1.5,2.5);
    \draw [>->] (1.875,1.875)--(2.625,0.625) node[midway,right]{1} ;

    \draw[solid] (1.5,2.5)--(-1.5,2.5)  ;
    \draw [<-<] (-0.75,2.5)--(0.75,2.5)node[midway,above]{1} ;

    \draw[solid] (-1.5,2.5)--(-3,0) ;
    \draw [<-<] (-1.875,1.875)--(-2.625,0.625) node[midway,left]{2};

    \draw[solid] (-3,0)--(-1.5,-2.5);
    \draw [>->] (-1.875,-1.875)--(-2.625,-0.625)node[midway,left]{3};

    \draw[solid] (-1.5,-2.5)--(1.5,-2.5);
    \draw [<-<] (-0.75,-2.5)--(0.75,-2.5)node[midway,below]{1};

    \draw[solid] (1.5,-2.5)--(3,0);
    \draw [<-<] (1.875,-1.875)--(2.625,-0.625)node[midway,right]{2};

    \draw [densely dotted](3,0)--(5,0) ;
    \draw [densely dotted][<->] (3.5,0)--(4.5,0)node[midway,below]{1};

     \draw[solid] (1.5,2.5)--(3,4);
     \draw [<-<] (1.875,2.875)--(2.625,3.625)node[midway,right]{2};

    \draw [densely dotted](-1.5,2.5)--(-3,4) ;
    \draw [densely dotted][>-<] (-1.875,2.875)--(-2.625,3.625)node[midway,right]{3};

     \draw[solid] (-3,0)--(-5,0);
    \draw [>->] (-3.5,0)--(-4.5,0)node[midway,above]{1};

    \draw[densely dotted] (-1.5,-2.5)--(-3,-4);
    \draw [densely dotted][<->] (-1.875,-2.875)--(-2.625,-3.625)node[midway,left]{2};

     \draw[solid] (1.5,-2.5)--(3,-4);
    \draw [>->] (1.875,-2.875)--(2.625,-3.625)node[midway,left]{1};

   \fill (3,0) circle (.1);
   \fill (5,0) circle (.1);
   \fill (3,4) circle (.1);
   \node at (3,4)[above]{$y_{0}$};
   \fill (1.5,2.5) circle (.1);
   \node at (1.2,2){$x_{0}$};
   \fill (1.5,-2.5) circle (.1);
   \fill (3,-4) circle (.1);

   \fill (-3,0) circle (.1);
   \fill (-5,0) circle (.1);
   \fill (-3,4) circle (.1);
   \node at (-3,4)[above]{$y_{1}$};
   \fill (-1.5,2.5) circle (.1);
   \node at (-1,2){$x_{1}$};
   \fill (-1.5,-2.5) circle (.1);
   \fill (-3,-4) circle (.1);

  \end{tikzpicture}
     \caption{A $4$-NZF $f_{2}$ on $(CL_{6},\sigma_{2})$.}\label{CL6}
  \end{minipage}%
\end{figure}

      We claim that $(CL_{4k},\sigma)$ admits a $4$-NZF, where $k\geq 1$ is an integer.
      We consider a signed circular ladder $(CL_{4},\sigma_{1})$ which is isomorphic to $(CL_{4k},\sigma)$ if $k=1$.
      Fig. \ref{CL4} shows that $(CL_{4},\sigma_{1})$ admits a $4$-NZF $f_{1}$.
      Note that $\sigma_{1}(x_{0}x_{1})=\sigma_{1}(y_{0}y_{1})=+1$, $f_{1}(x_{0}x_{1})=\pm1$ and $f_{1}(y_{0}y_{1})=\pm2$.
      Therefore, the $(4(k-1),0)$-extender of $(CL_{4},\sigma_{1})$ admits a $4$-NZF, by Lemma \ref{e-f}.
      Note that, $(CL_{4k},\sigma)$ is isomorphic to the $(4(k-1),0)$-extender of $(CL_{4},\sigma_{1})$.
      Thus, $(CL_{4k},\sigma)$ admits a $4$-NZF.

      We claim that $(CL_{4k+2},\sigma)$ admits a $4$-NZF, where $k\geq 1$ is an integer.
      Consider a signed circular ladder $(CL_{6},\sigma_{2})$, as shown in Fig. \ref{CL6}.
      Additively, $(CL_{4k+2},\sigma)$ is isomorphic to the $(4(k-1),0)$-extender of $(CL_{6},\sigma_{2})$.
      Fig. \ref{CL6} shows that $(CL_{6},\sigma_{2})$ admits a $4$-NZF $f_{2}$ that satisfies the conditions of Lemma \ref{e-f}.
      Therefore, the $(4(k-1),0)$-extender of $(CL_{6},\sigma_{2})$ admits a $4$-NZF, and so does $(CL_{4k+2},\sigma)$.\\

     \textbf{Case 2.} $\sigma(C_{x})\neq\sigma(C_{y})$.

     Without loss of generality, assume that $\sigma(C_{x})=-1$ and $\sigma(C_{y})=+1$.
     Suppose that $(C_{x},\sigma)$ has only one negative edge, say $x_{0}x_{1}$, and $(C_{y},\sigma)$ is all-positive.
     Since $(CL_{n},\sigma)$ is flow-admissible, it follows that $n\geq2$.
     We shall consider two subcases with respect to the signs of $\sigma(x_{0}y_{0})$ and $\sigma(x_{1}y_{1})$.\\

     \textbf{Subcase 2.1.} $\sigma(x_{0}y_{0})=\sigma(x_{1}y_{1})$.

     Assume that $x_{0}y_{0}$ and $x_{1}y_{1}$ are positive.
     Otherwise, we switch at $V(C_{x})$.
     Then, there exists an all-positive Hamiltonian circuit $(C_{x}\cup C_{y}\cup\{x_{0}y_{0}, x_{1}y_{1}\})\setminus \{x_{0}x_{1}, y_{0}y_{1}\}$.
     By Theorem \ref{BH-6-flow}, $(CL_{n},\sigma)$ admits a $6$-NZF.\\

     \textbf{Subcase 2.2.} $\sigma(x_{0}y_{0})\neq\sigma(x_{1}y_{1})$.

     Without loss of generality, assume that $\sigma(x_{0}y_{0})=-1$ and $\sigma(x_{1}y_{1})=+1$.
     We claim that there is another negative edge in $M\setminus \{x_{0}y_{0}, x_{1}y_{1}\}$.
     Otherwise, we switch at $x_{0}$ such that $E(CL_{n})$ has only one negative edge, leading to a contradiction.
     Hence, there exists an $i\in [2,n-1]$ such that $\sigma(x_{i}y_{i})\neq\sigma(x_{i-1}y_{i-1})$.
     Without loss of generality, assume $\sigma(x_{i}y_{i})=-1$ and $\sigma(x_{i-1}y_{i-1})=+1$.
     Then we switch at $\{x_{0},x_{n-1},x_{n-2}\cdots x_{i} \}$, and denote the resulting signed graph by $(CL_{n},\sigma')$.
     In $(CL_{n},\sigma')$, $C_{x}$ has only one negative edge $x_{i}x_{i-1}$, $C_{y}$ remains all-positive, and $\sigma'(x_{i}y_{i})=\sigma'(x_{i-1}y_{i-1})=+1$.
     Then there exists an all-positive Hamiltonian circuit $(C_{x}\cup C_{y}\cup\{x_{i}y_{i}, x_{i-1}y_{i-1}\})\setminus \{x_{i}x_{i-1}, y_{i}y_{i-1}\}$ in $(CL_{n},\sigma')$.
     By Theorem \ref{BH-6-flow}, $(CL_{n},\sigma')$ admits a $6$-NZF, and so does $(CL_{n},\sigma)$.\\

     \textbf{Case 3.} $\sigma(C_{x})=\sigma(C_{y})=-1$.

     Without loss of generality, suppose that $(C_{x},\sigma)$ has only one negative edge $x_{0}x_{1}$ and $(C_{y},\sigma)$ has only one negative edge $y_{0}y_{1}$.
     Since $(CL_{n},\sigma)$ is flow-admissible, it follows that $n\geq1$.
     If $n=1$, then $(CL_{n},\sigma)$ is isomorphic to a long barbell.
     Thus, $(CL_{n},\sigma)$ admits a $3$-NZF if $n=1$.
     Now, consider $n\geq2$.\\

      \textbf{Subcase 3.1.} There exists an $i \in [0, n-1]$ such that $\sigma(x_{i}y_{i})=\sigma(x_{i+1}y_{i+1})$ modulo $n$.

      Assume that $\sigma(x_{i}y_{i})=\sigma(x_{i+1}y_{i+1})=+1$; otherwise, perform a switching at $V(C_{x})$.
      Define $H=x_{i+1}x_{i+2}\cdots x_{n-1}x_{0}x_{1}\cdots x_{i}y_{i}y_{i-1}\cdots y_{0}y_{n-1} y_{n-2}\cdots y_{i+1}x_{i+1}$.
      It is easy to verify that $H$ forms a Hamiltonian circuit of $CL_{n}$.
      Additionally, there are only two negative edges $x_{0}x_{1}$ and $y_{0}y_{1}$, in $(H,\sigma)$.
      Thus, $(H,\sigma)$ is a balanced Hamiltonian circuit in $CL_{n}$.
      By Theorem \ref{BH-6-flow}, $(CL_{n},\sigma)$ admits a $6$-NZF.\\

      \textbf{Subcase 3.2.} There is no $i \in [0, n-1]$ such that $\sigma(x_{i}y_{i})\neq\sigma(x_{i+1}y_{i+1})$ modulo $n$.

      It is evident that $n$ is even.
      Assume that $\sigma(x_{1}y_{1})=+1$;
      otherwise, perform a switching at $V(C_{1})$.

      We claim that $(CL_{4k},\sigma)$ admits a $6$-NZF, where $k\geq 1$ is an integer.
      Fig. \ref{-CL4} shows that the signed circular ladder $(CL_{4},\sigma_{3})$ admits a $6$-NZF $f_{3}$ that satisfies the conditions of Lemma \ref{e-f}.
      Therefore, the $(4(k-1),3)$-extender of $(CL_{4},\sigma_{3})$ admits a $6$-NZF.
      Additionally, $(CL_{4k},\sigma)$ is isomorphic to the $(4(k-1),3)$-extender of $(CL_{4},\sigma_{3})$.
      Thus, $(CL_{4k},\sigma)$ admits a $6$-NZF.

   \begin{figure}[h]
  \centering
  \begin{minipage}[b]{.5\textwidth}
  \centering
  \begin{tikzpicture}[scale=0.7]


    \draw[densely dotted] (-2,2)--(2,2) ;
    \draw [densely dotted][<->] (-1,2)--(1,2)node[midway,above]{1} ;

    \draw[solid] (-2,2)--(-2,-2) ;
    \draw [>->] (-2,1)--(-2,-1) node[midway,left]{2};

    \draw[solid] (-2,-2)--(2,-2);
    \draw [>->] (-1,-2)--(1,-2)node[midway,below]{2};

    \draw[solid] (2,2)--(2,-2);
    \draw [>->] (2,1)--(2,-1)node[midway,right]{2};

    \draw[densely dotted] (-1,1)--(1,1);
    \draw [densely dotted][>-<] (-0.5,1)--(0.5,1) node[midway,above]{2};

    \draw[solid] (-1,1)--(-1,-1);
    \draw [<-<] (-1,0.5)--(-1,-0.5)node[midway,left]{3};

    \draw[solid] (-1,-1)--(1,-1) ;
    \draw [<-<] (-0.5,-1)--(0.5,-1) node[midway,below]{5};

    \draw[solid] (1,1)--(1,-1);
    \draw [>->] (1,0.5)--(1,-0.5)node[midway,right]{1};

    \draw [solid](-2,2)--(-1,1) node[left,midway]{1};
    \draw [<-<] (-1.75,1.75)--(-1.25,1.25) ;

     \draw[densely dotted] (-2,-2)--(-1,-1);
    \draw [densely dotted][>-<] (-1.75,-1.75)--(-1.25,-1.25)node[left]{2} ;

    \draw [solid](2,-2)--(1,-1) node[left,midway]{4} ;
    \draw [>->] (1.75,-1.75)--(1.25,-1.25);

     \draw[densely dotted] (2,2)--(1,1);
    \draw [densely dotted][<->] (1.75,1.75)--(1.25,1.25)node[right]{1};

   \fill (-2,2) circle (.1);
   \node at (-2,2)[left]{$y_{1}$};
   \fill (-2,-2) circle (.1);
   \fill (2,-2) circle (.1);
   \node at (2,-2)[right]{$y_{3}$};
   \fill (2,2) circle (.1);
   \node at (2,2)[right]{$y_{0}$};

    \fill (-1,1) circle (.1);
   \node at (-1,0.7)[right]{$x_{1}$};
   \fill (-1,-1) circle (.1);
   \fill (1,-1) circle (.1);
   \node at (1,-0.7)[left]{$x_{3}$};
   \fill (1,1) circle (.1);
   \node at (1,0.7)[left]{$x_{0}$};

  \end{tikzpicture}
   \caption{A $6$-NZF $f_{3}$ on $(CL_{4},\sigma_{3})$.}\label{-CL4}
   \end{minipage}%
   \begin{minipage}[b]{.5\textwidth}
   \centering
    \begin{tikzpicture}[scale=0.5]

    \draw [densely dotted](-3,0) arc (180:225:3);
    \draw [densely dotted](3,0) arc (0:-45:3);
    \draw (3,0) arc (0:45:3);
    \draw (-3,0) arc (180:133:3);
    \draw [densely dotted](1.5,0) arc (0:-180:1.5);
    \draw (1.5,0) arc (0:180:1.5);

    \draw[<-<] (2.12,2.12) arc (45:135:3);
    \draw (0,3)--(0,3)node[midway,above]{2};
    \draw[densely dotted][<->] (2.12,-2.12) arc (-45:-135:3);
    \draw (0,-3)--(0,-3)node[midway,above]{1};

    \draw[<-<](1.06,1.06) arc (45:135:1.5);
    \draw (0,1.5)--(0,1.5)node[midway,above]{1};
    \draw[densely dotted][<->] (1.06,-1.06) arc (-45:-135:1.5);
    \draw (0,-1.5)--(0,-1.5)node[midway,above]{2};

    \draw [densely dotted](3,0)--(1.5,0)node[midway,above]{3};
    \draw [densely dotted][>-<] (1.875,0)--(2.625,0);
    \draw (-3,0)--(-1.5,0)node[midway,above]{1};
    \draw [>->] (-1.875,0)--(-2.625,0);

   \fill (3,0) circle (.1);
   \fill (-3,0) circle (.1);
   \fill (1.5,0) circle (.1);
   \fill (-1.5,0) circle (.1);

     \end{tikzpicture}
      \caption{A $4$-NZF on $(CL_{2},\sigma)$.}\label{CL2}
      \end{minipage}
      \begin{minipage}[b]{.5\textwidth}
  \centering
  \begin{tikzpicture}[scale=0.4]

    \draw[solid] (5,0)--(3,4) ;
    \draw [>->] (3.5,3)--(4.5,1)node[right,above,midway]{1};

    \draw[solid] (3,4)--(-3,4) ;
    \draw [>->] (-1.5,4)--(1.5,4) node[above,midway]{2} ;

    \draw[solid] (-3,4)--(-5,0);
    \draw [<-<] (-3.5,3)--(-4.5,1)node[left,midway]{3};

    \draw[solid] (-5,0)--(-3,-4);
    \draw [>->] (-3.5,-3)--(-4.5,-1)node[left,midway]{2} ;

    \draw[densely dotted] (-3,-4)--(3,-4);
    \draw [densely dotted][<->] (-1.5,-4)--(1.5,-4)node[midway,below]{1} ;

    \draw[solid] (3,-4)--(5,0);
    \draw [<-<] (3.5,-3)--(4.5,-1)node[midway,right]{2} ;

    \draw[solid] (3,0)--(1.5,2.5);
    \draw [>->] (1.875,1.875)--(2.625,0.625) node[midway,right]{2} ;

    \draw[solid] (1.5,2.5)--(-1.5,2.5)  ;
    \draw [>->] (-0.75,2.5)--(0.75,2.5)node[midway,above]{3} ;

    \draw[solid] (-1.5,2.5)--(-3,0) ;
    \draw [<-<] (-1.875,1.875)--(-2.625,0.625) node[midway,left]{2};

    \draw[solid] (-3,0)--(-1.5,-2.5);
    \draw [>->] (-1.875,-1.875)--(-2.625,-0.625)node[midway,left]{1};

    \draw [densely dotted] (-1.5,-2.5)--(1.5,-2.5);
    \draw [densely dotted][<->] (-0.75,-2.5)--(0.75,-2.5)node[midway,below]{2};

    \draw[solid] (1.5,-2.5)--(3,0);
    \draw [<-<] (1.875,-1.875)--(2.625,-0.625)node[midway,right]{1};

    \draw [solid](3,0)--(5,0) ;
    \draw [solid][>->] (3.5,0)--(4.5,0)node[midway,below]{1};

     \draw[densely dotted] (1.5,2.5)--(3,4);
     \draw [densely dotted][>-<] (1.875,2.875)--(2.625,3.625)node[midway,right]{1};

    \draw [solid](-1.5,2.5)--(-3,4) ;
    \draw [solid][>-<] (-1.875,2.875)--(-2.625,3.625)node[midway,right]{1};

     \draw[densely dotted] (-3,0)--(-5,0);
    \draw [densely dotted][<->] (-3.5,0)--(-4.5,0)node[midway,above]{1};

    \draw[solid] (-1.5,-2.5)--(-3,-4);
    \draw [solid][>->] (-1.875,-2.875)--(-2.625,-3.625)node[midway,left]{1};

     \draw[densely dotted] (1.5,-2.5)--(3,-4);
    \draw [densely dotted][>-<] (1.875,-2.875)--(2.625,-3.625)node[midway,left]{3};

   \fill (3,0) circle (.1);
   \fill (5,0) circle (.1);
   \fill (3,4) circle (.1);
   \fill (1.5,2.5) circle (.1);
   \fill (1.5,-2.5) circle (.1);
   \node at (1.5,-2)[left]{$x_{0}$};
   \fill (3,-4) circle (.1);
   \node at (3,-4)[right]{$y_{0}$};

   \fill (-3,0) circle (.1);
   \node at (-3,0)[right]{$x_{2}$};
   \fill (-5,0) circle (.1);
   \node at (-5,0)[left]{$y_{2}$};
   \fill (-3,4) circle (.1);
   \fill (-1.5,2.5) circle (.1);
   \fill (-1.5,-2.5) circle (.1);
   \node at (-1.5,-2)[right]{$x_{1}$};
   \fill (-3,-4) circle (.1);
   \node at (-3,-4)[left]{$y_{1}$};

  \end{tikzpicture}
     \caption{A $4$-NZF on $(CL_{6},\sigma_{4})$.}\label{-CL6}
  \end{minipage}%
\end{figure}

      We claim that $(CL_{4k+2},\sigma)$ admits a $4$-NZF, where $k\geq 0$ is an integer.
      For $k=0$, Fig. \ref{CL2} shows that $(CL_{4k+2},\sigma)$ admits a $4$-NZF.
      Now, suppose that $k\geq1$.
      Fig. \ref{-CL6} illustrates that $(CL_{6},\sigma_{4})$ admits a $4$-NZF $f_{4}$.
      By Lemma \ref{e-f}, the $(4(k-1),1)$-extender of $(CL_{6},\sigma_{4})$ admits a $4$-NZF.
      Since $(CL_{4k+2},\sigma)$ is isomorphic to the $(4(k-1),1)$-extender of $(CL_{6},\sigma_{4})$, it follows that $(CL_{4k+2},\sigma)$ admits a $4$-NZF.
   \end{proof}

\subsection{Flow number of signed Cayley graphs on
abelian groups of odd order.}
\

    In this subsection, we characterize the flow number of flow-admissible signed Cayley graphs on abelian groups of odd order.
    In order to present this characterization, we also characterize the flow number of flow-admissible signed Hamilton-decomposable graphs.
    A graph is termed \textit{Hamilton-decomposable} if it can be decomposed into several edge-disjoint Hamiltonian circuits.

    Recall that the \textit{flow number} of $(G,\sigma)$, denoted by $\Phi(G,\sigma)$, is the minimum $k$ such that $(G,\sigma)$ admits a $k$-NZF.
    The main result of this subsection is as follows.

    \begin{thm}\label{c-h-dec}
     Let $A$ be a finite abelian group of odd order and $\Gamma=Cay(A,S)$ is connected.
     If $(\Gamma,\sigma)$ is flow-admissible, then

      {\rm(1)} $\Phi(\Gamma,\sigma)=2$ if and only if $\left| E_{N}(\Gamma,\sigma)\right|$ is even;

      {\rm(2)} $\Phi(\Gamma,\sigma)=3$ if and only if $\left| E_{N}(\Gamma,\sigma)\right|$ is odd and $\frac{\left| S \right|}{2}\geq 3$;

      {\rm(3)} $\Phi(\Gamma,\sigma)=4$ if and only if $\left| E_{N}(\Gamma,\sigma)\right|$ is odd and $\frac{\left| S \right|}{2}=2$.
    \end{thm}

   Let $A$ be an abelian group of odd order.
   By Lagrange's Theorem, for any $x\in A$, the order of $x$ is odd.
   Thus, there is no element $x\in A$ such that $x^{2}=1_{e}$, meaning there are no involutions in $A$.
   Consider the Cayley graph $\Gamma=Cay(A,S)$.
   Since $\left| A \right| $ is odd and $S$ is closed under taking inverses, it follows that $\left| S \right|$ is even.
   Therefore, $\Gamma=Cay(A,S)$ is an even graph.
   If $\Gamma$ is connected, then $\Gamma$ is Eulerian.
   Thus, the Cayley graph $\Gamma$ discussed in Theorem \ref{c-h-dec} is Eulerian.
   Consequently, the following result is necessary.

    \begin{thm}\cite{MS17}\label{euler}
      Let $(G,\sigma)$ be a signed Eulerian graph.
      Then

      {\rm(1)} $(G,\sigma)$ has no nowhere-zero flow if and only if $(G,\sigma)$ is unbalanced and $(G \setminus e,\sigma)$ is balanced for some edge $e$;

      {\rm(2)} $\Phi(G,\sigma) = 2$ if and only if $(G,\sigma)$ has an even number of negative edges;

      {\rm(3)} $\Phi(G,\sigma) = 3$ if and only if $(G,\sigma)$ can be decomposed into three Eulerian subgraphs, with an odd number of negative edges each, that share a common vertex;

      {\rm(4)} $\Phi(G,\sigma) = 4$ otherwise.
    \end{thm}

    Alspach \cite{A84} conjectured that any $2k$-regular connected Cayley graph on an abelian group has a Hamiltonian decomposition.
    Westlund et al. \cite{WLK09} validated Alspach's conjecture for the case \(k=3\), under the condition that the abelian group has an odd order.

    \begin{thm}\cite{WLK09} \label{3-h-dec}
    Every connected $6$-regular Cayley graph on an abelian group of odd order is decomposable into three Hamiltonian circuits.
    \end{thm}

    Thus, to prove Theorem \ref{c-h-dec}, we need to characterize the flow number of flow-admissible signed Hamilton-decomposable graphs.
    If $G$ can be decomposed into $l$ edge-disjoint Hamiltonian circuits, then $G$ is $2l$-edge-connected.
    Therefore, the following result is necessary.

    \begin{thm}\cite{WYZZ14}\label{3-NZF}
      Every flow-admissible $8$-edge-connected signed graph admits a nowhere-zero $3$-flow.
    \end{thm}

   A path $P$ is referred to as an \textit{xy-path} if it connects the vertices $x$ and $y$.
   The characterization of the flow number of flow-admissible signed Hamilton-decomposable graphs is as follows.

   \begin{thm}\label{h-dec}
    Let graph $G$ be $2k$-regular and Hamilton-decomposable.
     If $(G,\sigma)$ is flow-admissible, then

      {\rm(1)} $\Phi(G,\sigma)=2$ if and only if $\left| E_{N}(G,\sigma)\right|$ is even;

      {\rm(2)} $\Phi(G,\sigma)=3$ if and only if $\left| E_{N}(G,\sigma)\right| $ is odd and $k\geq 3$;

      {\rm(3)} $\Phi(G,\sigma)=4$ if and only if $\left| E_{N}(G,\sigma)\right|$ is odd and $k=2$.
  \end{thm}

    \begin{proof}
      Statement (1) follows directly as a corollary of Theorem \ref{euler}.
      If $(G,\sigma)$ contains an odd number of negative edges, then according to Theorem \ref{euler}, $3\leq\Phi(G,\sigma)\leq4$.
      Since a $4$-regular graph cannot be decomposed into three Eulerian subgraphs, this confirms the validity of Statement (3).
      According to Theorem \ref{3-NZF}, if $(G,\sigma)$ has an odd number of negative edges and $k>3$, then $(G,\sigma)$ admits a $3$-NZF because $G$ is $8$-edge-connected.
      Therefore,  it is sufficient to prove that $\Phi(G,\sigma)=3$ when $\left| E_{N}(G,\sigma)\right| $ is odd and $k= 3$.

      Consider three edge-disjoint Hamiltonian circuits $C^{1}$, $C^{2}$, and $C^{3}$ in $G$, such that their edge sets satisfy $E(C^{1})\cup E(C^{2})\cup E(C^{3})=E(G)$.
      If all circuits in $\{C^{1}, C^{2}, C^{3}\}$ are unbalanced, then $(G,\sigma)$ admits a $3$-NZF by Statement (3) of Theorem \ref{euler}.
      If only two circuits in $\{C^{1}, C^{2}, C^{3}\}$ are unbalanced, then $(G,\sigma)$ has an even number of negative edges, leading to a contradiction.
      Therefore, it remains to prove that if there is only one unbalanced circuit, say $C_{1}$, within $\{C^{1}, C^{2}, C^{3}\}$, then $(G,\sigma)$ admits a $3$-NZF.
      Without loss of generality, assume that $(C^{2},\sigma)$ is all-positive; if not, we switch at certain vertices of $C^{2}$ to make all its edges positive.
      We will consider two cases based on the signature of $(C^{3},\sigma)$.\\

      \textbf{Case 1.} $(C^{3},\sigma)$ is not all-positive.

      There exists a negative edge $e$ within $E(C^{3})$, and $\left| E_{N}(C^{3},\sigma)\right|$ is even.
      Let the ends of $e$ be $u$ and $v$.
      Since $C^{2}$ is a Hamiltonian circuit, it can be decomposed into two edge-disjoint $uv$-paths $P^{2}_{a}$ and $P^{2}_{b}$.
      Consequently, $P^{2}_{a}\cup \{e\}$ and $P^{2}_{b}\cup (C^{3}\setminus e)$ form two Eulerian subgraphs of $G$, each containing an odd number of negative edges.
      Thus, $(G,\sigma)$ can be decomposed into three Eulerian subgraphs $P^{2}_{a}\cup  e$, $P^{2}_{b}\cup (C^{3}\setminus e)$, and $C^{1}$, each having an odd number of negative edges and sharing the common vertices $u$ and $v$.
      Therefore, $\Phi(G,\sigma)=3$ by Statement (3) of Theorem \ref{euler}.\\

      \textbf{Case 2.} $(C^{3},\sigma)$ is all-positive.

      Since $(G,\sigma)$ is flow-admissible and $\left| E_{N}(G,\sigma)\right|$ is odd, there are at least three negative edges in $E_{N}(G,\sigma)$.
      Moreover, because both $(C^{2},\sigma)$ and $(C^{3},\sigma)$ are all-positive, there are at least three negative edges in $(C^{1},\sigma)$.
      Let $e_{1}$ and $e_{2}$ denote two negative edges in $(C^{1},\sigma)$.
      Let the ends of $e_{1}$ be $u_{1}$ and $v_{1}$.
      Then, in $C^{2}$, there exist two $u_{1}v_{1}$-paths, denoted by $P^{2}_{\alpha}$ and $P^{2}_{\beta}$.
      Let the ends of $e_{2}$ be $u_{2}$ and $v_{2}$.
      Thus, in $C^{3}$, there exist two $u_{2}v_{2}$-paths, denoted by $P^{3}_{\gamma}$ and $P^{3}_{\delta}$.
      Since $C^{3}$ is a Hamiltonian circuit of $G$, one of paths $P^{3}_{\gamma}$ or $P^{3}_{\delta}$ contains $u_{1}$ as a vertex, say $P^{3}_{\delta}$.
      Consequently, $(G,\sigma)$ can be decomposed into three Eulerian subgraphs: $(C^{1}\setminus \{e_{1},e_{2}\})\cup P^{2}_{\alpha} \cup P^{3}_{\gamma}$, $P^{2}_{\beta} \cup \{e_{1}\}$, and $P^{3}_{\delta}\cup \{e_{2}\}$.
      Each subgraph contains an odd number of negative edges and share a common vertex $u_{1}$.
      Therefore, by Statement (3) of Theorem \ref{euler}, $\Phi(G,\sigma)=3$.
    \end{proof}

   Now, we can complete the proof of Theorem  \ref{c-h-dec}.

    \begin{proof}[Proof of Theroem \ref{c-h-dec}]

      Given that $\Gamma$ is Eulerian, $\Phi(G,\sigma)=2$ if and only if $\left| E_{N}(\Gamma,\sigma)\right|$ is even, by Theorem \ref{euler}.
      Thus, Statement (1) holds.

      If $(\Gamma,\sigma)$ contains an odd number of negative edges, then $3\leq\Phi(G,\sigma)\leq4$, according to  Theorem \ref{euler}.
      Given that $\Gamma$ is $\left|S\right|$-regular,  it cannot be decomposed into three Eulerian subgraphs when $\frac{\left|S\right|}{2}=2$.
      Thus, Statement (3) holds.

      According to Theorem \ref{3-NZF}, if $(\Gamma,\sigma)$ has an odd number of negative edges and $\frac{\left|S\right|}{2}>3$, then $(\Gamma,\sigma)$ admits a $3$-NZF because $\Gamma$ is $8$-edge-connected.
      Therefore, it is sufficient to consider cases where $\left| E_{N}(\Gamma,\sigma)\right| $ is odd and $\frac{\left|S\right|}{2}= 3$.
      According to Theorem \ref{3-h-dec}, $(\Gamma,\sigma)$ is a flow-admissible signed Hamilton-decomposable graph.
      Thus, $(\Gamma,\sigma)$ admits a $3$-NZF because the Statement (2) of Theorem \ref{h-dec}.
    \end{proof}

\section*{Acknowledgements}

This work is supported by National Natural Science Foundation of China (Grant No 12461006), Guizhou Provincial Basic Research Program (Grant No. ZD[2025]085) and Scientic Research Foundation of Guizhou University(Grant No. [2023]41).

\section*{Declarations}
{\bf Conflict of interest} The authors declare that they have no conflict of interest.


\begin{thebibliography}{99}
\bibitem{A84}
    B. Alspach, {\rm Research Problem 59}, {\it Discrete Math.}, 50(1984) 115.



\bibitem{BM08}
    J. A. Bondy, U. S. R. Murty, {\rm Graph theory}, {\it Springer}, 2008.

\bibitem{B83}
    A. Bouchet, {\rm Nowhere-zero integral flows on bidirected graph}, {\it J. Combin. Theory Ser. B}, 34(1983) 279--292.

\bibitem{CLLZ18}
    J. Cheng, Y. Lu, R. Luo, C.-Q. Zhang, {\rm Signed graphs: from modulo flows to integer-valued flows}, {\it SIAM J. Discrete Math.}, 32(2018) 956--965.

\bibitem{DLLLZZ21}
     M. DeVos, J. Li, Y. Lu, R. Luo, C.-Q. Zhang, Z. Zhang, {\rm Flows on flow-admissible signed graphs}, {\it J. Combin. Theory Ser. B}, 149(2021) 198--221.

\bibitem{J88}
     F. Jaeger, {\rm Nowhere-zero flow problems, in Selected Topics in Graph Theory III}, {\it Academic Press, San Diego, CA}, 1988 71-95.

\bibitem{LLLZZ23}
    L. Li, C. Li, R. Luo, C-Q. Zhang, H. Zhang, {\rm Flows of $3$-edge-colorable cubic signed graphs}, {\it Eur. J. Combin.}, 108(2023) 103627.


\bibitem{LMSZ25}
    R. Luo, E. M\'{a}\v{c}ajov\v{a}, M. \v{S}koviera, and C-Q. Zhang, {\rm An 8-flow theorem for signed graphs}, {\it SIAM J. Discrete Math.}, 39(2025) 1409--1417.

\bibitem{MS15}
    E. M\'{a}\v{c}ajov\v{a}, M. \v{S}koviera, {\rm Remarks on nowhere-zero flows in signed cubic graphs}, {\it Discrete Math.}, 338(2015) 809--815.

\bibitem{MS17}
    E. M\'{a}\v{c}ajov\v{a}, M. \v{S}koviera, {\rm Nowhere-zero flows on signed Eulerian graphs}, {\it SIAM J. Discrete Math.}, 31(2017) 1937--1952.

\bibitem{NS09}
    M. N\'{a}n\'{a}siov\'{a}, M. \v{S}koviera, {\rm Nowhere-zero flows in Cayley graphs and Sylow 2-subgroups}, {\it J. Algebraic Comb.}, 30(2009) 103--110.

\bibitem{PSS05}
    P. Poto\v{c}nik, M. \v{S}koviera and R. \v{S}krekovski, {\rm Nowhere-zero 3-flows in abelian Cayley graphs}, {\it Discrete Math.}, 297(2005) 119--127.

\bibitem{RZ11}
     A. Raspaud, X. Zhu, {\rm Circular flow on signed graphs}, {\it J. Combin. Theory Ser. B}, 101(2011) 464--479.

\bibitem{SS15}
     M. Schubert, E. Steffen, {\rm Nowhere-zero flows on signed regular graphs}, {\it Eur. J. Combin.}, 48(2015) 34--47.

\bibitem{T49}
     W. T. Tutte, {\rm On the embedding of linear graphs in surfaces}, {\it Proc. London Math. Soc.}, 51(1949) 474--483.

\bibitem{T54}
     W. T. Tutte, {\rm A contribution to the theory of chromatic polynomial}, {\it Canad. J. Math.}, 6(1954) 80--91.


\bibitem{WYZZ14}
    Y. Wu, D. Ye, W. Zang, and C-Q. Zhang, {\rm Nowhere-zero 3-flows in signed graphs}, {\it SIAM J. Discrete Math.}, 28(2014) 1628--1637.

\bibitem{WLK09}
    E. E. Westlund, J. Liu, D.L. Kreher, {\rm $6$-regular Cayley graphs on Abelian groups of odd order are hamiltonian decomposable}, {\it Discrete Math.}, 309(2009) 5106--5110.

\bibitem{Z97}
    C.-Q. Zhang, {\rm Integer Flows and Cycle Covers of Graphs}, {\it Marcel Dekker}, New York, 1997.

\bibitem{Z87}
    O. Z\'{y}ka, {\rm Nowhere-zero 30-flow on bidirected graphs (Thesis)}, {\it Charles University}, Praha, 1987, KAM-DIMATIA Series.




\end{thebibliography}
\end{document}